\documentclass[12pt]{amsart}
\UseRawInputEncoding
\usepackage{fancybox}
\usepackage{shuffle}
\usepackage{amsmath}
\usepackage{amsfonts}
\usepackage{amssymb}
\usepackage[all]{xy}           
\usepackage{bm}
\usepackage{bbding}
\usepackage{txfonts}
\usepackage{amscd}

\usepackage{xspace}
\usepackage[shortlabels]{enumitem}
\usepackage{ifpdf}

\ifpdf
  \usepackage[colorlinks,final,backref=page,hyperindex]{hyperref}
\else
  \usepackage[colorlinks,final,backref=page,hyperindex,hypertex]{hyperref}
\fi
\usepackage{tikz}
\usepackage[active]{srcltx}

\makeatletter

\newtheorem{thm}{Theorem}[section]
\newtheorem{lem}[thm]{Lemma}
\newtheorem{cor}[thm]{Corollary}
\newtheorem{pro}[thm]{Proposition}
\newtheorem{ex}[thm]{Example}
\theoremstyle{definition}
\newtheorem{rmk}[thm]{Remark}
\newtheorem{defi}[thm]{Definition}

\newcommand{\nc}{\newcommand}
\newcommand{\delete}[1]{}

\nc{\mlabel}[1]{\label{#1}}  
\nc{\mcite}[1]{\cite{#1}}  
\nc{\mref}[1]{\ref{#1}}  
\nc{\mbibitem}[1]{\bibitem{#1}} 

\delete{
\nc{\mlabel}[1]{\label{#1}{\hfill \hspace{1cm}{\bf{{\ }\hfill(#1)}}}}
\nc{\mcite}[1]{\cite{#1}{{\em{{\ }(#1)}}}}  
\nc{\mref}[1]{\ref{#1}{{\em{{\ }(#1)}}}}  
\nc{\mbibitem}[1]{\bibitem[\em #1]{#1}} 
}

\newcommand {\emptycomment}[1]{}

\nc{\oprn}{\theta}
\nc{\Oprn}{\Theta}

\nc{\calo}{\mathcal{O}}
\nc{\oop}{$\mathcal{O}$-operator\xspace}
\nc{\oops}{$\mathcal{O}$-operators\xspace}
\nc{\mrho}{{\bm{\varrho}}}
\nc{\emk}{\mathbf{K}}
\nc{\invlim}{\displaystyle{\lim_{\longleftarrow}}\,}
\nc{\ot}{\otimes}

\newcommand{\be }{\begin{equation}}
\newcommand{\ee }{\end{equation}}

\newcommand{\g}{\mathfrak g}
\newcommand{\h}{\mathfrak h}
\newcommand{\m}{\mathfrak m}

\newcommand{\huaB}{\mathcal{B}}

\newcommand{\huaL}{\mathcal{L}}
\newcommand{\huaR}{\mathcal{R}}

\newcommand{\huaC}{{\mathcal{C}}}

\newcommand{\huaH}{\mathcal{H}}

\newcommand{\huaO}{{\mathcal{O}}}

\newcommand{\huaZ}{\mathcal{Z}}

\newcommand{\frkC}{\mathfrak C}

\newcommand{\frkL}{\mathfrak L}

\newcommand{\frkR}{\mathfrak R}

\newcommand{\frkT}{\mathfrak T}
\newcommand{\frkU}{\mathfrak U}

\newcommand{\frkX}{\mathfrak X}

\newcommand{\Courant}[1]{\left\llbracket  #1\right\rrbracket }

\newcommand{\Id}{{\rm{Id}}}

\newcommand{\br}[1]{   [ \cdot,    \cdot  ]   }
\newcommand{\ltp}[1]{\Courant{\cdot,\cdot,\cdot}}

\newcommand{\Hom}{\mathrm{Hom}}

\newcommand{\Ob}{\mathsf{Ob}}

\newcommand{\gl}{\mathfrak {gl}}

\newcommand{\ad}{\mathrm{ad}}

\nc{\CV}{\mathbf{C}}

\newcommand{\LYA}{Lie-Yamaguti algebra}

\begin{document}

\title[Post Lie-Yamaguti algebras, relative Rota-Baxter operators of nonzero weights]{Post Lie-Yamaguti algebras, relative Rota-Baxter operators of nonzero weights, and their deformations}

\author{Jia Zhao}
\address{Jia Zhao, School of Sciences, Nantong University, Nantong, 226019, Jiangsu, China}
\email{zhaojia@ntu.edu.cn}

\author{Senrong Xu}
\address{Senrong Xu, School of Mathematical Sciences, Jiangsu University, Zhenjiang, 212013, Jiangsu, China}
\email{senrongxu@ujs.edu.cn}

\author{Yu Qiao*}
\address{Yu Qiao (corresponding author), School of Mathematics and Statistics, Shaanxi Normal University, Xi'an, 710119, Shaanxi, China}
\email{yqiao@snnu.edu.cn}

\date{\today}

\thanks{{\em Keywords}: relative Rota-Baxter operato of weight $1$, post-Lie-Yamaguti algebra, cohomology, deformation}

\thanks{{\em Mathematics Subject Classification} (2020): Primary 17B38; Secondary 17B60,17A99}


\begin{abstract}
 In this paper, we introduce the notions of relative Rota-Baxter operators of weight $1$ on Lie-Yamaguti algebras, and post-\LYA s, which is an underlying algebraic structure of relative Rota-Baxter operators of weight $1$. We give the relationship between these two algebraic structures. Besides, we establish the cohomology theory of relative Rota-Baxter operators of weight $1$ via the Yamaguti cohomology. Consequently, we use this cohomology to characterize linear deformations of relative Rota-Baxter operators of weight $1$ on Lie-Yamaguti algebras. We show that if two linear deformations of a relative Rota-Baxter operator of weight $1$ are equivalent, then their infinitesimals are in the same cohomology class in the first cohomology group. Moreover, we show that an order $n$ deformation of a relative Rota-Baxter operator of weight $1$ can be extended to an order $n+1$ deformation if and only if the obstruction class in the second cohomology group is trivial.
\end{abstract}



\maketitle


\smallskip


\tableofcontents

\allowdisplaybreaks

 \section{Introduction}
 G. Baxter introduced the notion of Rota-Baxter operator on associative algebras when studying fluctuation theory \cite{Ba}. Later, Kupershmidt introduced relative Rota-Baxter operators (also called $\huaO$-operators or Kupershmidt operators) on Lie algebras when studying the classical Yang-Baxter equation \cite{Kupershmidt}. Rota-Baxter operators have many applications in mathematical physics such as integrability systems and quantum mechanics \cite{Bai-Bellier-Guo-Ni}. Since then, Rota-Baxter algebras has become a hot topic in many cases, for example the fist work of Rota-Baxter algebras appeared recently \cite{Gub}. Besides, several works of relative Rota-Baxter operators on Lie algebras and other kinds of algebras were investigated in \cite{TBGS,THS,T.S2,ZQ1}. Moreover, Sheng and the first author introduced the notion of relative Rota-Baxter operators on \LYA s and give its relation with symplectic structures on \LYA s \cite{Sheng Zhao}. Consequently, the first author and the corresponding author examined cohomology and deformations of relative Rota-Baxter operators on \LYA s in \cite{ZQ2}.

 In general, relative Rota-Baxter operators of weight $1$ on Lie algebras corresponds to the solutions to the modified classical Yang-Baxter equation, and plays an important role in generalized Lax pairs and affine structures \cite{BGN}. Bai and his collaborators put relative Rota-Baxter operators of nonzero weights and the modified classical Yang-Baxter equation together by introducing the notion of extended relative Rota-Baxter operators, and gave some applications in Lie bialgebras \cite{BGN}. A relative Rota-Baxter operators of weight $1$ gives rise to a post-Lie algebra, which is a key object in differential geometry. A post-Lie algebra can be seen as a noncommutative pre-Lie algebra. Recently, relative Rota-Baxter operators of nonzero weights on $3$-Lie algebras and $3$-post-Lie algebras were introduced in \cite{HSZ}, and authors explored its cohomology by using the controlling algebras. Moreover, Zhou and his collaborators examined cohomology and homotopy theory of Rota-Baxter algebras of nonzero weights in \cite{WZ}, and cohomology, deformations, and extensions of differential algebras of nonzero weights were studied in \cite{GLSZ}.

\subsection{Lie-Yamaguti algebras}
 A \LYA ~was dated back to Nomizu's work on the affine invariant connections on homogeneous spaces in 1950's \cite{Nomizu}.  Later in 1960's, Yamaguti introduced an algebraic structure and called it a general Lie triple system or a Lie triple algebra \cite{Yamaguti1,Yamaguti2,Yamaguti3}, since it can be seen as a generalization of a Lie algebra and a Lie triple system. Kinyon and Weinstein first called this object a \LYA~  when studying Courant algebroids in the earlier 21st century \cite{Weinstein}. Since a \LYA ~stems from differential geometry completely and it is a key higher structure in mathematical physics, it has attracted much attention and is widely investigated recently. For instance, Benito and his collaborators deeply explored irreducible Lie-Yamaguti algebras and their relations with orthogonal Lie algebras \cite{B.B.M,B.D.E,B.E.M1,B.E.M2}. Deformations and extensions of Lie-Yamaguti algebras were examined in \cite{L.CHEN,Ma Y,Zhang1,Zhang2}. Sheng, the first author, and Zhou analyzed product structures and complex structures on Lie-Yamaguti algebras by means of Nijenhuis operators in \cite{Sheng Zhao}. Takahashi studied modules over quandles using representations of Lie-Yamaguti algebras in \cite{Takahashi}.
\subsection{Deformations}
Mathematical physics has many branches in mathematics and can be applied in Lie theory and representation theory \cite{IK,LiuPeiXia}. Deformation theory plays an important role in mathematics and mathematical physics. A deformation of a mathematical object, roughly speaking, means that it preserves its original structure after a parameter perturbation. In physics, deformation theory comes from quantizing classical mechanics, and this idea promotes some researches on quantum groups in mathematics \cite{CP,Hart}. Recently, deformation quantization has produced many elegant works in the context of mathematical physics. Based on work of complex analysis by Kodaira and Spencer \cite{Kodaira}, deformation theory was generalized in algebra \cite{Hart}. Deformation of algebra can be dated back to works on associative algebra by Gerstenhaber \cite{Gerstenhaber1,Gerstenhaber2,Gerstenhaber3,Gerstenhaber4,Gerstenhaber5}. Later, Nijenhuis and Richadson studied deformations on Lie algebra \cite{Nij1}. Balavoine generalized deformation theory to operads \cite{bala}.

In the context of algebras, deformation has close connection with cohomology. For instance, a suitable cohomology can be used to characterize deformations. In particular, a linear deformation of Lie algebras is controlled by a second cohomolpogy group; an order $n$ deformation can be extended to an order $n+1$ deformation if and only if its obstruction class is trivial; a trivial deformation gives rise to a Nijenhuis operator \cite{Dor}, which plays an important role in deformation theory and has applications in integrability of constructing biHamiltonian systems \cite{Dor}. Sheng and his collaborators have a series of works on deformation theory on (3-)Lie algebras. For example, they studied deformations on 3-Lie algebras and even $n$-Lie algebras \cite{LSZB} and examined product and complex structures on $3$-Lie algebras using Nijenhuis operators \cite{T.S}. Moreover, they construct a controlling algebra that characterizes deformations of relative Rota-Baxter operators on Lie algebras,  on $3$-Lie algebras, and on Leibniz algebras respectively \cite{TBGS,THS,T.S2}. Recently, Pei and his colleagues established crossed homomorphisms on Lie algebras via the same methods, and generalized constructions of many kinds of Lie algebras by using bifunctors \cite{PSTZ}. Besides, the first author and the corresponding author investigated cohomology and linear deformations of $\mathsf{LieYRep}$ pairs and explored several properties of relative Rota-Baxter-Nijenhuis structures on $\mathsf{LieYRep}$ pairs in \cite{ZQ1}, and cohomology and deformations of relative Rota-Baxter operators on \LYA s in \cite{ZQ2}.
\subsection{Motivation}
It is nature to consider relative Rota-Baxter operators of weight $1$ on \LYA s. The first purpose of the present paper is to devote its relations between relative Rota-Baxter operators of weight $1$ and associated algebraic structures. For this purpose, we introduce the notion of post-\LYA s, which can be seen as the underlying algebraic structures of relative Rota-Baxter operators of weight $1$. And we show that a post-\LYA ~gives rise to a \LYA ~structure and an action of this \LYA ~on itself. Consequently, the identity map is a relative Rota-Baxter operator of weight $1$ on this induced \LYA.

The second purpose is to explore cohomology and deformations of relative Rota-Baxter operators of weight $1$. For this aim, we need to construct a new representation of the induced sub-adjacent \LYA ~by relative Rota-Baxter operators of weight $1$, and to construct an $0$-cochain, which is one of the main tasks in this paper. Once the cohomology is established, we are able to examine deformations of relative Rota-Baxter operators of weight $1$. We intend to study two kinds of deformations: linear and higher order deformations.
\subsection{Outline of the paper}
The paper is structured as follows. In Section 2, we recall some basic notions such as \LYA s, representations, and cohomology. In Section 3, we introduce the notion of relative Rota-Baxter operators of weight $1$ on \LYA s, and study some properties of it. In Section 4, we introduce the notion of post-\LYA s, which is the underlying algebraic structures of relative Rota-Baxter operators of weight $1$. We give the relationship between post-\LYA s and \LYA s. In Section 5, we establish the cohomology of relative Rota-Baxter operators of weight $1$ on \LYA s, and examine a functorial property of the cohomology theory. Moreover, we use this type of cohomology to explore two kinds of deformations, and show that the infinitesimal of linear deformations can be governed by cohomology and that the extension of a higher deformation is characterized by a special cohomology class, which is called the obstruction class.

In this paper, all vector spaces are assumed to be over a field $\mathbb{K}$ of characteristic $0$ and finite-dimensional.

\section{Preliminaries: Lie-Yamaguti algebras, representations and cohomology}
In this section, we recall some basic notions such as \LYA s, representations and their cohomology theories.
The notion of  Lie-Yamaguti algebras was introduced by  Yamaguti in \cite{Yamaguti1}.

\begin{defi}\cite{Weinstein}\label{LY}
A {\bf \LYA} is a vector space $\g$ equipped with a bilinear bracket $[\cdot,\cdot]:\wedge^2  \mathfrak{g} \to \mathfrak{g} $ and a trilinear bracket $\Courant{\cdot,\cdot,\cdot}:\wedge^2\g \otimes  \mathfrak{g} \to \mathfrak{g} $, which meet the following conditions: for all $x,y,z,w,t \in \g$,
\begin{eqnarray}
~ &&\label{LY1}[[x,y],z]+[[y,z],x]+[[z,x],y]+\Courant{x,y,z}+\Courant{y,z,x}+\Courant{z,x,y}=0,\\
~ &&\label{LY2}\Courant{[x,y],z,w}+\Courant{[y,z],x,w}+\Courant{[z,x],y,w}=0,\\
~ &&\label{LY3}\Courant{x,y,[z,w]}=[\Courant{x,y,z},w]+[z,\Courant{x,y,w}],\\
~ &&\Courant{x,y,\Courant{z,w,t}}=\Courant{\Courant{x,y,z},w,t}+\Courant{z,\Courant{x,y,w},t}+\Courant{z,w,\Courant{x,y,t}}.\label{fundamental}
\end{eqnarray}
In the sequel, we denote a \LYA ~by $(\g,\br,\ltp))$.
\end{defi}

\begin{ex}
Let $(\g,\br))$ be a Lie algebra. Define a trilinear bracket
$$\ltp::\wedge^2\g\ot \g\to \g$$
by
$$\Courant{x,y,z}:=[[x,y],z],\quad \forall x,y,z \in \g.$$
Then by a direct computation, we know that $(\g,\br,\ltp))$ forms a \LYA.
\end{ex}

The following example is even more interesting.
\begin{ex}
Let $M$ be a closed manifold with an affine connection, and denote by $\frkX(M)$ the set of vector fields on $M$. For all $x, y, z \in \frkX(M) $, set
\begin{eqnarray*}
[x,y]&:=&-T(x,y),\\
\Courant{x,y,z}&:=&-R(x,y)z,
\end{eqnarray*}
where $T$ and $R$ are torsion tensor and curvature tensor respectively. It turns out that the triple
$ (\frkX(M),[\cdot,\cdot],\Courant{\cdot,\cdot,\cdot})$ forms a \LYA. See \cite{Nomizu} for more details.
\end{ex}
\emptycomment{
\begin{rmk}
Given a Lie-Yamaguti algebra $(\m,[\cdot,\cdot]_\m,\Courant{\cdot,\cdot,\cdot}_\m)$ and any two elements $x,y \in \m$, the linear map $D(x,y):\m \to \m,~z\mapsto D(x,y)z=\Courant{x,y,z}_\m$ is an (inner) derivation. Moreover, let $D(\m,\m)$ be the linear span of the inner derivations. Consider the vector space $\g(\m)=D(\m,\m)\oplus \m$, and endow it with a Lie bracket as follows: for all $x,y,z,t \in \m$
\begin{eqnarray*}
[D(x,y),D(z,t)]_{\g(\m)}&=&D(\Courant{x,y,z}_\m,t)+D(z,\Courant{x,y,t}_\m),\\
~[D(x,y),z]_{\g(\m)}&=&D(x,y)z=\Courant{x,y,z}_\m,\\
~[z,t]_{\g(\m)}&=&D(z,t)+[z,t]_\m.
\end{eqnarray*}
Then $(\g(\m),[\cdot,\cdot]_{\g(\m)})$ becomes a Lie algebra.
\end{rmk}}

Next, we recall the notion of representations of \LYA s.

\begin{defi}\cite{Yamaguti2}\label{defi:representation}
Let $(\g,[\cdot,\cdot],\Courant{\cdot,\cdot,\cdot})$ be a Lie-Yamaguti algebra. A {\bf representation} of $\g$ is a vector space $V$ equipped with a linear map $\rho:\g \to \gl(V)$ and a bilinear map $\mu:\otimes^2 \g \to \gl(V)$, which meet the following conditions: for all $x,y,z,w \in \g$,
\begin{eqnarray}
~&&\label{RLYb}\mu([x,y],z)-\mu(x,z)\rho(y)+\mu(y,z)\rho(x)=0,\\
~&&\label{RLYd}\mu(x,[y,z])-\rho(y)\mu(x,z)+\rho(z)\mu(x,y)=0,\\
~&&\label{RLYe}\rho(\Courant{x,y,z})=[D_{\rho,\mu}(x,y),\rho(z)],\\
~&&\label{RYT4}\mu(z,w)\mu(x,y)-\mu(y,w)\mu(x,z)-\mu(x,\Courant{y,z,w})+D_{\rho,\mu}(y,z)\mu(x,w)=0,\\
~&&\label{RLY5}\mu(\Courant{x,y,z},w)+\mu(z,\Courant{x,y,w})=[D_{\rho,\mu}(x,y),\mu(z,w)],
\end{eqnarray}
where the bilinear map $D_{\rho,\mu}:\otimes^2\g \to \gl(V)$ is given by
\begin{eqnarray}
 D_{\rho,\mu}(x,y):=\mu(y,x)-\mu(x,y)+[\rho(x),\rho(y)]-\rho([x,y]), \quad \forall x,y \in \g.\label{rep}
 \end{eqnarray}
It is obvious that $D_{\rho,\mu}$ is skew-symmetric, and we write $D$ in the sequel without ambiguities.  We denote a representation of $\g$ by $(V;\rho,\mu)$.
\end{defi}

\begin{rmk}\label{rmk:rep}
Let $(\g,[\cdot,\cdot],\Courant{\cdot,\cdot,\cdot})$ be a Lie-Yamaguti algebra and $(V;\rho,\mu)$ a representation of $\g$. If $\rho=0$ and the Lie-Yamaguti algebra $\g$ reduces to a Lie tripe system $(\g,\Courant{\cdot,\cdot,\cdot})$,  then the representation reduces to that of the Lie triple system $(\g,\Courant{\cdot,\cdot,\cdot})$: $(V;\mu)$. If $\mu=0$, $D=0$ and the Lie-Yamaguti algebra $\g$ reduces to a Lie algebra $(\g,[\cdot,\cdot])$, then the representation reduces to that of the Lie algebra $(\g,[\cdot,\cdot])$: $(V;\rho)$. Hence a representation of a Lie-Yamaguti algebra is a natural generalization of that of a Lie algebra or of a Lie triple system.
\end{rmk}

By a direct computation, we have the following lemma.
\begin{lem}
Suppose that $(V;\rho,\mu)$ is a representation of a Lie-Yamaguti algebra $(\g,[\cdot,\cdot],\Courant{\cdot,\cdot,\cdot})$. Then the following equalities are satisfied:
\begin{eqnarray*}
\label{RLYc}&&D([x,y],z)+D([y,z],x)+D([z,x],y)=0;\\
\label{RLY5a}&&D(\Courant{x,y,z},w)+D(z,\Courant{x,y,w})=[D(x,y),D_{\rho,\mu}(z,w)];\\
~ &&\mu(\Courant{x,y,z},w)=\mu(x,w)\mu(z,y)-\mu(y,w)\mu(z,x)-\mu(z,w)D(x,y).\label{RLY6}
\end{eqnarray*}
\end{lem}

\begin{ex}\label{ad}
Let $(\g,[\cdot,\cdot],\Courant{\cdot,\cdot,\cdot})$ be a Lie-Yamaguti algebra. We define linear maps $\ad:\g \to \gl(\g)$ and $\frkR :\otimes^2\g \to \gl(\g)$ by $x \mapsto \ad_x$ and $(x,y) \mapsto \mathfrak{R}_{x,y}$ respectively, where $\ad_xz=[x,z]$ and $\mathfrak{R}_{x,y}z=\Courant{z,x,y}$ for all $z \in \g$. Then $(\ad,\mathfrak{R})$ forms a representation of $\g$ on itself, where $\frkL:= D_{\ad,\frkR}$ is given by
\begin{eqnarray*}
\frkL_{x,y}=\mathfrak{R}_{y,x}-\mathfrak{R}_{x,y}+[\ad_x,\ad_y]-\ad_{[x,y]}, \quad \forall x,y \in \g.
\end{eqnarray*}
By \eqref{LY1}, we have
\begin{eqnarray*}
\frkL_{x,y}z=\Courant{x,y,z}, \quad \forall z \in \g.\label{lef}
\end{eqnarray*}
In this case, $(\g;\ad,\frkR)$ is called the {\bf adjoint representation} of $\g$.
\end{ex}
\emptycomment{
The representations of Lie-Yamaguti algebras can be characterized by the semidirect Lie-Yamaguti algebras. This fact is revealed via the following proposition.

\begin{pro}\cite{Zhang1}
Let $(\g,[\cdot,\cdot],\Courant{\cdot,\cdot,\cdot})$ be a Lie-Yamaguti algebra and $V$ a vector space. Let $\rho:\g \to \gl(V)$ and $\mu:\otimes^2 \g \to \gl(V)$ be linear maps. Then $(V;\rho,\mu)$ is a representation of $(\g,[\cdot,\cdot],\Courant{\cdot,\cdot,\cdot})$ if and only if there is a Lie-Yamaguti algebra structure $([\cdot,\cdot]_{\rho,\mu},\Courant{\cdot,\cdot,\cdot}_{\rho,\mu})$ on the direct sum $\g \oplus V$ which is defined by for all $x,y,z \in \g, ~u,v,w \in V$,
\begin{eqnarray}
\label{semi1}[x+u,y+v]_{\rho,\mu}&=&[x,y]+\rho(x)v-\rho(y)u,\\
\label{semi2}~\Courant{x+u,y+v,z+w}_{\rho,\mu}&=&\Courant{x,y,z}+D_{\rho,\mu}(x,y)w+\mu(y,z)u-\mu(x,z)v,
\end{eqnarray}
where $D_{\rho,\mu}$ is given by \eqref{rep}.
This Lie-Yamaguti algebra $(\g \oplus V,[\cdot,\cdot]_{\rho,\mu},\Courant{\cdot,\cdot,\cdot}_{\rho,\mu})$ is called the {\bf semidirect product Lie-Yamaguti algebra}, and is denoted by $\g \ltimes_{\rho,\mu} V$.
\end{pro}
\begin{proof}
The proof is a direct computation, so we omit the details. Or one can see \cite{Zhang1} for more details.
\end{proof}}

Let us recall the cohomology theory on Lie-Yamaguti algebras given in \cite{Yamaguti2}. Let $(\g,[\cdot,\cdot],\Courant{\cdot,\cdot,\cdot})$ be a  Lie-Yamaguti algebra and $(V;\rho,\mu)$ a representation of $\g$. We denote the set of $p$-cochains by $C^p_{\rm LieY}(\g,V)~(p \geqslant 1)$, where
\begin{eqnarray*}
C^{n+1}_{\rm LieY}(\g,V)\triangleq
\begin{cases}
\Hom(\underbrace{\wedge^2\g\otimes \cdots \otimes \wedge^2\g}_n,V)\times \Hom(\underbrace{\wedge^2\g\otimes\cdots\otimes\wedge^2\g}_{n}\otimes\g,V), & \forall n\geqslant 1,\\
\Hom(\g,V), &n=0.
\end{cases}
\end{eqnarray*}

In the sequel, we recall the coboundary map of $p$-cochains:
\begin{itemize}
\item If $n\geqslant 1$, for any $(f,g)\in C^{n+1}_{\rm LieY}(\g,V)$, the coboundary map is given by
$$\delta=(\delta_{\rm I},\delta_{\rm II}):C^{n+1}_{\rm LieY}(\g,V)\to C^{n+2}_{\rm LieY}(\g,V),$$
$$\qquad \qquad\qquad \qquad\qquad \quad (f,g)\mapsto(\delta_{\rm I}(f,g),\delta_{\rm II}(f,g))$$
where $\delta_{\rm I}(f,g)$ and $\delta_{\rm II}(f,g)$ are defined to be
\begin{eqnarray}
~\nonumber &&\Big(\delta_{\rm I}(f,g)\Big)(\frkX_1,\cdots,\frkX_{n+1})\\
~ &=&(-1)^n\Big(\rho(x_{n+1})g(\frkX_1,\cdots,\frkX_n,y_{n+1})-\rho(y_{n+1})g(\frkX_1,\cdots,\frkX_n,x_{n+1})\label{cohomology1}\\
~\nonumber &&\quad\quad-g(\frkX_1,\cdots,\frkX_n,[x_{n+1},y_{n+1}])\Big)\\
~\nonumber &&+\sum_{k=1}^{n}(-1)^{k+1}D(\frkX_k)f(\frkX_1,\cdots,\widehat{\frkX_k},\cdots,\frkX_{n+1})\\
~\nonumber &&+\sum_{1\leqslant k<l\leqslant n+1}(-1)^{k}f(\frkX_1,\cdots,\widehat{\frkX_k},\cdots,\frkX_k\circ\frkX_l,\cdots,\frkX_{n+1}),
\end{eqnarray}
and
\begin{eqnarray}
~\nonumber &&\Big(\delta_{\rm II}(f,g)\Big)(\frkX_1,\cdots,\frkX_{n+1},z)\\
~ &=&(-1)^n\Big(\mu(y_{n+1},z)g(\frkX_1,\cdots,\frkX_n,x_{n+1})-\mu(x_{n+1},z)g(\frkX_1,\cdots,\frkX_n,y_{n+1})\Big)\label{cohomology2}\\
~\nonumber &&+\sum_{k=1}^{n+1}(-1)^{k+1}D(\frkX_k)g(\frkX_1,\cdots,\widehat{\frkX_k},\cdots,\frkX_{n+1},z)\\
~\nonumber &&+\sum_{1\leqslant k<l\leqslant n+1}(-1)^kg(\frkX_1,\cdots,\widehat{\frkX_k},\cdots,\frkX_k\circ\frkX_l,\cdots,\frkX_{n+1},z)\\
~\nonumber &&+\sum_{k=1}^{n+1}(-1)^kg(\frkX_1,\cdots,\widehat{\frkX_k},\cdots,\frkX_{n+1},\Courant{x_k,y_k,z}),
\end{eqnarray}
respectively. Here $\frkX_i=x_i\wedge y_i\in\wedge^2\g~(i=1,\cdots,n+1),~z\in \g$ and $\frkX_k\circ\frkX_l:=\Courant{x_k,y_k,x_l}\wedge y_l+x_l\wedge\Courant{x_k,y_k,y_l}$.

\item If $n=0$, for any $f \in C^1_{\rm LieY}(\g,V)$, the coboundary map is given by
$$\delta:C^1_{\rm LieY}(\g,V)\to C^2_{\rm LieY}(\g,V),$$
$$\qquad \qquad \qquad f\mapsto (\delta_{\rm I}(f),\delta_{\rm II}(f))$$
where
\begin{eqnarray}
\label{cohomology3}\Big(\delta_{\rm I}(f)\Big)(x,y)&=&\rho(x)f(y)-\rho(y)f(x)-f([x,y]),\\
~ \label{cohomology4}\Big(\delta_{\rm II}(f)\Big)(x,y,z)&=&D(x,y)f(z)+\mu(y,z)f(x)-\mu(x,z)f(y)-f(\Courant{x,y,z}),\quad \forall x,y, z\in \g.
\end{eqnarray}
\end{itemize}

Yamaguti showed the following basic fact.

\begin{pro}{\rm \cite{Yamaguti2}}
 With the notations above, for any $f\in C^1_{\rm LieY}(\g,V)$, we have
 \begin{eqnarray*}
 \delta_{\rm I}\Big(\delta_{\rm I}(f)),\delta_{\rm II}(f)\Big)=0\quad {\rm and} \quad\delta_{\rm II}\Big(\delta_{\rm I}(f)),\delta_{\rm II}(f)\Big)=0.
 \end{eqnarray*}
 Moreover, for all $(f,g)\in C^p_{\rm LieY}(\g,V),~(p\geqslant 2)$, we have
  \begin{eqnarray*}
  \delta_{\rm I}\Big(\delta_{\rm I}(f,g)),\delta_{\rm II}(f,g)\Big)=0\quad{\rm and} \quad \delta_{\rm II}\Big(\delta_{\rm I}(f,g)),\delta_{\rm II}(f,g)\Big) =0.
  \end{eqnarray*}
  Thus the cochain complex $(C^\bullet_{\rm LieY}(\g,V)=\bigoplus\limits_{p=1}^\infty C^p_{\rm LieY}(\g,V),\delta)$ is well defined. For convenience, we call this cohomology the {\bf Yamaguti cohomology} in this paper.
  \end{pro}

\begin{defi}
With the above notations, let $(f,g)$ in $C^p_{\rm LieY}(\g,V))$ (resp. $f\in C^1_{\rm LieY}(\g,V)$ for $p=1$) be a $p$-cochain. If it satisfies $\delta(f,g)=0$ (resp. $\delta(f)=0$), then it is called a $p$-cocycle. If there exists $(h,s)\in C^{p-1}_{\rm LieY}(\g,V)$,~(resp. $t\in C^1(\g,V)$, if $p=2$) such that $(f,g)=\delta(h,s)$~(resp. $(f,g)=\delta(t)$), then it is called a $p$-coboundary ($p\geqslant 2$). The set of $p$-cocycles and that of $p$-coboundaries are denoted by $Z^p_{\rm LieY}(\g,V)$ and $B^p_{\rm LieY}(\g,V)$ respectively. The resulting $p$-cohomology group is defined to be the factor space
$$H^p_{\rm LieY}(\g,V)=Z^p_{\rm LieY}(\g,V)/B^p_{\rm LieY}(\g,V)\quad (p\geqslant2).$$
\end{defi}

\section{Relative Rota-Baxter operators of nonzero weights on Lie-Yamaguti algebras}
In this subsection, we introduce the notion of relative Rota-Baxter operators of weight $1$ on \LYA s.
Before this, we recall the notion of center of \LYA s, and introduce that of derived algebras of \LYA s.

Let $(\g,[\cdot,\cdot],\Courant{\cdot,\cdot,\cdot})$ be a \LYA. Recall in \cite{ZQX} that the {\bf center} of $\g$ is denoted by
$$C(\g):=\{x\in \g|[x,y]=0,\forall y\in \g\}\cap\Big(\{x\in \g|\Courant{x,y,z}=0,\forall y,z \in \g\}\cup\{x\in \g|\Courant{y,z,x}=0,\forall y,z \in \g\}\Big).$$
The subalgebra $[\g,\g]\cap\Courant{\g,\g,\g}$ of $\g$ is called the {\bf derived algebra} of $\g$, which is denoted by $\g^1$.

The following definition is also taken from \cite{ZQX}.
\begin{defi}
Let $(\g,[\cdot,\cdot]_\g,\Courant{\cdot,\cdot}_\g)$ and $(\h,[\cdot,\cdot]_\h,\Courant{\cdot,\cdot}_\h)$ be two Lie-Yamaguti algebras. Let $(\h;\rho,\mu)$ be a representation of $\g$ on the vector space $\h$, i.e., linear maps $\rho:\g\to\gl(\h)$, $\mu:\otimes^2\g\to\gl(\h)$, and $D:\wedge^2\g\to\gl(\g)$ are given by Eqs. \eqref{RLYb}-\eqref{rep}. If for all $x,y\in \g,~u,v,w\in \h$, the following conditions are satisfied
\begin{eqnarray}
\rho(x)u,\mu(x,y)u\in C(\h),\label{*1}\\
\rho(x)[u,v]_\h=\mu(x,y)[u,v]_\h=0,\label{*2}\\
\rho(x)\Courant{u,v,w}_\h=\mu(x,y)\Courant{u,v,w}_\h=0.\label{*3}
\end{eqnarray}
then we say that $(\rho,\mu)$ is an {\bf action} of $\g$ on $\h$.
\end{defi}

Let $(\rho,\mu)$ be an action of $\g$ on $\h$. By \eqref{rep}, we deduce that
$$D(x,y)u\in C(\g),~D(x,y)[u,v]_\h=D(x,y)\Courant{u,v,w}_\h=0,\quad\forall x,y\in \g,u,v,w\in \h.$$

\emptycomment{The following proposition shows that an action of \LYA s can be used to characterize semidirect product \LYA s.
\begin{pro}
Let $(\g,[\cdot,\cdot]_\g,\Courant{\cdot,\cdot}_\g)$ and $(\h,[\cdot,\cdot]_\h,\Courant{\cdot,\cdot}_\h)$ be two Lie-Yamaguti algebras. Let $(\rho,\mu)$ be an action of $\g$ on $\h$, then there is a Lie-Yamaguti algebra structure on the direct sum $\g\oplus\h$ defined by
\begin{eqnarray*}
[x+u,y+v]_{\rho,\mu}&=&[x,y]_\g+\rho(x)v-\rho(y)u+[u,v]_\h,\\
\Courant{x+u,y+v,z+w}_{\rho,\mu}&=&\Courant{x,y,z}_\g+D(x,y)w+\mu(y,z)u-\mu(x,z)v+\Courant{u,v,w}_\h,
\end{eqnarray*}
for all $x,y,z\in \g$ and $u,v,w\in \h$. This Lie-Yamaguti algebra is called the {\bf semidirect product Lie-Yamaguti algebra} with respect to the action $(\rho,\mu)$, and is denoted by $\g\ltimes_{\rho,\mu}\h$.
\end{pro}
\begin{proof}
It is a direct computation, and we omit the details.
\end{proof}}

The following definition is standard.

\begin{defi}\cite{Sheng Zhao,Takahashi}\label{homomorphism}
Suppose that $(\g,[\cdot,\cdot]_{\g},\Courant{\cdot,\cdot,\cdot}_{\g})$ and $(\h,[\cdot,\cdot]_{\h},\Courant{\cdot,\cdot,\cdot}_{\h})$ are two Lie-Yamaguti algebras. A {\bf homomorphism} from $(\g,[\cdot,\cdot]_{\g},\Courant{\cdot,\cdot,\cdot}_{\g})$ to $(\h,[\cdot,\cdot]_{\h},\Courant{\cdot,\cdot,\cdot}_{\h})$ is a linear map $\phi:\g \to \h$ that preserves the \LYA ~structures, that is, for all $x,y,z \in \g$,
\begin{eqnarray}
\phi([x,y]_{\g})&=&[\phi(x),\phi(y)]_{\h},\label{v1}\\
~ \phi(\Courant{x,y,z}_{\g})&=&\Courant{\phi(x),\phi(y),\phi(z)}_{\h}.\label{v2}
\end{eqnarray}
If, moreover, $\phi$ is a bijection, it is then called an {\bf isomorphism}.
\end{defi}

Now we are ready to introduce the notion of relative Rota-Baxter operators of nonzero weights on \LYA s.

\begin{defi}
Let $(\g,[\cdot,\cdot]_\g,\Courant{\cdot,\cdot}_\g)$ and $(\h,[\cdot,\cdot]_\h,\Courant{\cdot,\cdot}_\h)$ be two Lie-Yamaguti algebras. Let $(\rho,\mu)$ be an action of $\g$ on $\h$. A linear map $T:\h\longrightarrow\g$ is called a {\bf relative Rota-Baxter operator of weight $1$} from $(\h,[\cdot,\cdot]_\h,\Courant{\cdot,\cdot}_\h)$ to $(\g,[\cdot,\cdot]_\g,\Courant{\cdot,\cdot}_\g)$ with respect to the action $(\rho,\mu)$, if for all $u,v,w\in \h$,
\begin{eqnarray}
[Tu,Tv]_\g&=&T\Big(\rho(Tu)v-\rho(Tv)u+[u,v]_\h\Big),\label{chomo1}\\
\ \ \Courant{Tu,Tv,Tw}_\g&=&T\Big(D(Tu,Tv)w+\mu(Tv,Tw)u-\mu(Tu,Tw)v+\Courant{u,v,w}_\h\Big).\label{chomo2}
\end{eqnarray}
\end{defi}

Then we introduce the notion of homomorphisms of relative Rota-Baxter operators of nonzero weights.

\begin{defi}
Let $T$ and $T'$ be two relative Rota-Baxter operators of weight $1$ from a Lie-Yamaguti algebra $(\h,[\cdot,\cdot]_\h,\Courant{\cdot,\cdot,\cdot}_\h)$ to another Lie-Yamaguti algebra $(\g,[\cdot,\cdot]_\g,\Courant{\cdot,\cdot,\cdot}_\g)$ with respect to an action $(\rho,\mu)$. A {\bf homomorphism} from $T'$ to $T$ is a pair $(\psi_\g,\psi_\h)$, where $\psi_\g:\g\longrightarrow\g$ and $\psi_\h:\h\longrightarrow\h$ are two Lie-Yamaguti algebra homomorphisms such that
\begin{eqnarray}
\psi_\g\circ T'&=&T\circ \psi_\h,\label{crho1}\\
\psi_\h\Big(\rho(x)u\Big)&=&\rho\Big(\psi_\g(x)\Big)\psi_\h(u),\label{crosshomo1}\\
\psi_\h\Big(\mu(x,y)u\Big)&=&\mu\Big(\psi_\g(x),\psi_\g(y)\Big)\psi_\h(u),\quad\forall x,y\in \g, u\in \h.\label{crosshomo2}
\end{eqnarray}
In particular, if both $\psi_\g$ and $\psi_\h$ are invertible, then a homomorphism $(\psi_\g,\psi_h)$ becomes an {\bf isomorphism} from $T'$ to $T$.
\end{defi}

By Eqs. \eqref{crosshomo1} and \eqref{crosshomo2}, and a direct computation, we have the following proposition.

\begin{pro}
Let $T$ and $T'$ be two relative Rota-Baxter operators of weight $1$ from a Lie-Yamaguti algebra $(\h,[\cdot,\cdot]_\h,\Courant{\cdot,\cdot,\cdot}_\h)$ to another Lie-Yamaguti algebra $(\g,[\cdot,\cdot]_\g,\Courant{\cdot,\cdot,\cdot}_\g)$ with respect to an action $(\rho,\mu)$. Suppose that $(\psi_\g,\psi_\h)$ is a homomorphism from $T'$ to $T$, then we have
\begin{eqnarray}
\psi_\h\Big(D(x,y)u\Big)&=&D\Big(\psi_\g(x),\psi_\g(y)\Big)\psi_\h(u),\quad\forall x,y\in \g, u\in \h.\label{crosshomo3}
\end{eqnarray}
\end{pro}

The following proposition demonstrates that an action can be characterized by semidirect product \LYA s.

\begin{pro}
Let $(\rho,\mu)$ be an action of a \LYA ~$(\g,[\cdot,\cdot]_\g,\Courant{\cdot,\cdot,\cdot}_\g)$ on another \LYA ~$(\h,[\cdot,\cdot]_\h,\Courant{\cdot,\cdot,\cdot}_\h)$. Then there exists a \LYA ~structure $([\cdot,\cdot]_\ltimes,\Courant{\cdot,\cdot,\cdot}_\ltimes)$ on the direct sum $\g\oplus\h$ defined to be, for all $x,y,z\in \g,u,v,w\in \h$
\begin{eqnarray*}
[x+u,y+v]_\ltimes&=&[x,y]_\g+\rho(x)v-\rho(y)u+[u,v]_\h,\\
~\Courant{x+u,y+v,z+w}_\ltimes&=&\Courant{x,y,z}_\g+D(x,y)w+\mu(y,z)u-\mu(x,z)v+\Courant{u,v,w}_\h.
\end{eqnarray*}
This \LYA ~$(\g\oplus\h,[\cdot,\cdot]_\ltimes,\Courant{\cdot,\cdot,\cdot}_\ltimes)$ is called the {\bf semidirect \LYA} with respect to an action $(\rho,\mu)$, which is denoted by $\g\ltimes_{\rho,\mu}\h$.
\end{pro}
\begin{proof}
It only involves computations, and we omit the details.
\end{proof}

\begin{thm}
Let $(\g,[\cdot,\cdot]_\g,\Courant{\cdot,\cdot,\cdot}_\g)$ and $(\h,[\cdot,\cdot]_\h,\Courant{\cdot,\cdot,\cdot}_\h)$ be two \LYA s, and $(\rho,\mu)$ an action of $\g$ on $\h$. Then a linear map $T:\h\longrightarrow\g$ is a relative Rota-Baxter operator of weight $1$ with respect to $(\rho,\mu)$ if and only if the graph of $T$
$$Gr(T):=\{Tu+u:u\in \h\}$$
is a subalgebra of the semidirect \LYA ~$\g\ltimes_{\rho,\mu}\h$.
\end{thm}
\begin{proof}
Let $T:\h\longrightarrow\g$ be a linear map. Then for all $u,v,w\in \h$, we have
\begin{eqnarray*}
~ &&[Tu+u,Tv+v]_\ltimes\\
~ &=&[Tu,Tv]_\g+\rho(Tu)v-\rho(Tv)u+[u,v]_\h,
\end{eqnarray*}
and
\begin{eqnarray*}
~&&\Courant{Tu+u,Tv+v,Tw+w}_\ltimes\\
~ &=&\Courant{Tu,Tv,Tw}_\g+D(Tu,Tv)w+\mu(Tv,Tw)u-\mu(Tu,Tw)v+\Courant{u,v,w}_\h.
\end{eqnarray*}
Thus the graph of $T$ is a subalgebra of $\g\ltimes\h$ if and only if
\begin{eqnarray*}
[Tu,Tv]_\g&=&T\Big(\rho(Tu)v-\rho(Tv)u+[u,v]_\h\Big),\\
~ \Courant{Tu,Tv,Tw}_\g&=&\Big(D(Tu,Tv)w+\mu(Tv,Tw)u-\mu(Tu,Tw)v+\Courant{u,v,w}_\h\Big),
\end{eqnarray*}
which implies that $T:\h\longrightarrow\g$ is a relative Rota-Baxter operator of weight $1$. This completes the proof.
\end{proof}

In the sequel, we give the relationship between relative Rota-Baxter operator of weight $1$ and Nijenhuis operators. Recall that in \cite{Sheng Zhao}, Sheng and the first author et. al gave the notion of Nijenhuis operators on \LYA s. Let $(\g,[\cdot,\cdot]_\g,\Courant{\cdot,\cdot,\cdot}_\g)$ be a \LYA. A linear map $N:\g\longrightarrow\g$ is called a {\bf Nijenhuis operator}, if the following equalities hold
\begin{eqnarray*}
[Nx,Ny]_\g&=&N\Big([Nx,y]_\g+[x,Ny]_\g-N[x,y]_\g\Big),\\
~\Courant{Nx,Ny,Nz}_\g&=&N\Big(\Courant{Nx,Ny,z}_\g+\Courant{Nx,y,Nz}_\g+\Courant{x,Ny,Nz}_\g\\
~ &&-N\Courant{Nx,y,z}_\g-N\Courant{x,Ny,z}_\g-N\Courant{x,y,Nz}_\g\\
~ &&+N^2\Courant{x,y,z}_\g\Big), \quad\forall x,y,z\in \g.
\end{eqnarray*}

\begin{pro}
Let $(\g,[\cdot,\cdot]_\g,\Courant{\cdot,\cdot,\cdot}_\g)$ and $(\h,[\cdot,\cdot]_\h,\Courant{\cdot,\cdot,\cdot}_\h)$ be two \LYA s, and $(\rho,\mu)$ an action of $\g$ on $\h$. Then a linear map $T:\h\longrightarrow\g$ is a relative Rota-Baxter operator of weight $1$ with respect to $(\rho,\mu)$ if and only if
$$\overline{T}=
\begin{pmatrix}
\Id & T \\
0 & 0
\end{pmatrix}$$
is a Nijenhuis operator on the semidirect \LYA ~$\g\ltimes_{\rho,\mu}\h$.
\end{pro}
\begin{proof}
For all $x,y,z\in \g,~u,v,w\in \h$, on one hand, we have
\begin{eqnarray*}
~ &&[\overline{T}(x+u),\overline{T}(y+v)]_\ltimes=[x+Tu,y+Tv]_\g\\
~ &=&[x,y]_\g+[x,Tv]_\g+[Tu,y]_\g+[Tu,Tv]_\g,
\end{eqnarray*}
and
\begin{eqnarray*}
~ &&\Courant{\overline{T}(x+u),\overline{T}(y+v),\overline{T}(z+w)}_\ltimes=\Courant{x+Tu,y+Tv,z+Tw}_\g\\
~ &=&\Courant{x,y,z}_\g+\Courant{Tu,Tv,z}_\g+\Courant{Tu,y,Tw}_\g+\Courant{x,Tv,Tw}_\g\\
~ &&+\Courant{Tu,y,z}_\g+\Courant{x,Tv,z}_\g+\Courant{x,y,Tw}_\g+\Courant{Tu,Tv,Tw}_\g.
\end{eqnarray*}
On the other hand, since $\overline{T}^2=\overline{T}$, we have
\begin{eqnarray*}
~ &&\overline{T}\Big([\overline{T}(x+u),y+v]_\ltimes+[x+u,\overline{T}(y+v)]_\ltimes-\overline{T}[x+u,y+v]_\ltimes\Big)\\
~ &=&[x,y]_\g+[x,Tv]_\g+[Tu,y]_\g+T\Big(\rho(Tu)v-\rho(Tv)u+[u,v]_\h\Big),
\end{eqnarray*}
and
\begin{eqnarray*}
~ &&\overline{T}\Big(\Courant{\overline{T}(x+u),\overline{T}(y+v),z+w}_\ltimes+\Courant{\overline{T}(x+u),y+v,\overline{T}(z+w)}_\ltimes
+\Courant{x+u,\overline{T}(y+v),\overline{T}(z+w)}_\ltimes\Big)\\
~ &&-\overline{T}^2\Big(\Courant{\overline{T}(x+u),y+v,z+w}_\ltimes+\Courant{x+u,\overline{T}(y+v),z+w}_\ltimes
+\Courant{x+u,y+v,\overline{T}(z+w)}_\ltimes\Big)\\
~ &&+\overline{T}^3\Courant{x+u,y+v,z+w}_\ltimes\\
~ &=&\Courant{x,y,z}_\g+\Courant{Tu,Tv,z}_\g+\Courant{Tu,y,Tw}_\g+\Courant{x,Tv,Tw}_\g+\Courant{Tu,y,z}_\g+\Courant{x,Tv,z}_\g+\Courant{x,y,Tw}_\g\\
~ &&+T\Big(D(Tu,Tv)w+\mu(Tv,Tw)u-\mu(Tu,Tw)v+\Courant{u,v,w}_\h\Big).
\end{eqnarray*}
Thus $\overline{T}$ is a Nijenhuis operator on the semidirect \LYA ~$\g\ltimes\h$ if and only if $T$ is a relative Rota-Baxter operator of weight $1$ from $\h$ to $\g$. This finishes the proof.
\end{proof}

\begin{pro}
Let $(\g,[\cdot,\cdot]_\g,\Courant{\cdot,\cdot,\cdot}_\g)$ be a \LYA ~such that the adjoint representation $(\g;\ad,\frkR)$ is an action of the \LYA ~$\g$ on itself. Let $\h$ be an abelian subalgebra of $\g$ such that $\g^1\cap\h=0$. Suppose that $\mathfrak{t}$ is a compliment of $\h$, i.e., $\g=\mathfrak{t}\oplus\h$ as vector spaces. Then the projection $P:\g\longrightarrow\g$ onto the subspace $\h$ is a relative Rota-Baxter operator of weight $1$ from $\g$ to $\g$ with respect to the action $(\ad,\frkR)$.
\end{pro}
\begin{proof}
For all $x,y,z\in \g$, denote by $\bar x,\bar y,\bar z$ the images of $x,y,z$ by the projection $P$. Since $\h$ is abelian and $\g^1\cap\h=0$, we have
\begin{eqnarray*}
~ &&[Px,Py]_\g-P\Big([Px,y]_\g-[x,Py]_\g-[x,y]_\g\Big)\\
~ &=&[\bar x,\bar y]_\g-P\Big([\bar x,y]_\g-[x,\bar y]_\g-[x,y]_\g\Big)\\
~ &=&0.
\end{eqnarray*}
Similarly, we also have that
\begin{eqnarray*}
~ &&\Courant{Px,Py,Pz}_\g-P\Big(\Courant{Px,Py,z}_\g+\Courant{x,Py,Pz}_\g-\Courant{y,Px,Pz}_\g+\Courant{x,y,z}_\g\Big)\\
~ &&\Courant{\bar x,\bar y,\bar z}_\g-P\Big(\Courant{\bar x,\bar y,z}_\g+\Courant{x,\bar y,\bar z}_\g-\Courant{y,\bar x,\bar z}_\g+\Courant{x,y,z}_\g\Big)\\
~ &=&0.
\end{eqnarray*}
Thus $P$ is a relative Rota-Baxter operator of weight $1$ from $\g$ to $\g$ with respect to the action $(\ad,\frkR)$.
\end{proof}

At the end of this section, we give an example of relative Rota-Baxter operator of weight $1$ on \LYA s.
\begin{ex}
Let $(\g,[\cdot,\cdot],\Courant{\cdot,\cdot,\cdot})$ be a $4$-dimensional \LYA, and $\{e_1,e_2,e_3,e_4\}$ a basis. The nonzero brackets are given by
$$[e_1,e_2]=2e_4,\quad\quad \Courant{e_1,e_2,e_1}=e_4.$$
It is obvious that the center of $\g$ is spanned by $\{e_3,e_4\}$, and that the adjoint representation $(\g;\ad,\frkR)$ is an action of $\g$ on itself. Let $\h$ be an abelian subalgebra spanned by $\{e_1,e_2\}$. Then by Proposition, the projection $P:\g\longrightarrow\g$ defined to be
\begin{eqnarray*}
\begin{cases}
P(e_1)=e_1,\\
P(e_2)=e_2,\\
P(e_3)=0,\\
P(e_4)=0,
\end{cases}
\end{eqnarray*}
is a relative Rota-Baxter operator of weight $1$ from $\g$ to $\g$ with respect to the action $(\ad,\frkR)$.
\end{ex}

\section{Post Lie-Yamaguti algebras}
In this section, we introduce the notion of post-\LYA s, which is the underlying algebraic structure of relative Rota-Baxter operators of nonzero weights. Moreover, we show that there exists a \LYA ~structure on a post-\LYA ~and that an action can also be obtained via the given post-\LYA.

\begin{defi}
A {\bf post-\LYA} is a vector space $A$ endowed with two binary operations $\cdot:\wedge^2A\longrightarrow A,~*:\otimes^2A\longrightarrow A$ and two ternary operations $\langle\cdot,\cdot,\cdot\rangle:\wedge^2A\otimes A\longrightarrow A,~ \{\cdot,\cdot,\cdot\}:\otimes^3A\longrightarrow A$ such that $(A,\cdot,\langle\cdot,\cdot,\cdot\rangle)$ is a \LYA ~and the following equalities hold for all $x,y,z,w,t\in A$,
\begin{eqnarray}
~ \{z,[x,y]_C,w\}&=&\{y*z,x,w\}-\{x*z,y,w\},\label{post1}\\
~ \{x,y,[z,w]_c\}&=&z*\{x,y,w\}-w*\{x,y,z\},\\
~ \Courant{x,y,z}_C*w&=&\{x,y,z*w\}_D-z*\{x,y,w\}_D,\label{post3}\\
~ \{x,y,\Courant{z,w,t}_C\}&=&\{\{x,w,z\},w,t\}-\{\{x,y,w\},z,t\}+\{z,w,\{x,y,t\}_D \},\\
~ \{x,y,\{z,w,t\}_D\}&=&\{\{x,y,z\}_D,w,t\}+\{z,\Courant{x,y,w}_C,t\}+\{z,w,\Courant{x,y,t}_C\},\label{post5}\\
~  (x*y)\cdot z&=&\langle x*y,z,w\rangle=\langle z,w,x*y\rangle=0,\label{post6}\\
~  x*(yz)&=&\{xy,z,w\}=0,\\
~  x*\langle z,w,t\rangle&=&\{\langle z,w,t\rangle, x,y\}=0,\label{post8}
\end{eqnarray}
where operations $[\cdot,\cdot]_C,~\Courant{\cdot,\cdot,\cdot}_C$ and $\{\cdot,\cdot,\cdot\}_D$ are defined to be
\begin{eqnarray}
[x,y]_C&=&x*y-y*x+xy,\label{sub1}\\
~\Courant{x,y,z}_C&=&\{x,y,z\}_D+\{x,y,z\}-\{y,x,z\}+\langle x,y,z\rangle,\label{sub2}\\
~ \{x,y,z\}_D&=&\{z,y,x\}-\{z,x,y\}+(y,x,z)-(x,y,z)-(xy)*z,\quad\forall x,y,z\in A.\label{dbra}
\end{eqnarray}
Here the notation $(\cdot,\cdot,\cdot)$ stands for the associator with respect to the operation $*$, i.e., for all $x,y,z\in A,$
$$(x,y,z)=(x*y)*z-x*(y*z).$$
\end{defi}

\begin{rmk}
Let $(A,\cdot,*,\{\cdot,\cdot,\cdot\},\langle\cdot,\cdot,\cdot\rangle)$ be a post-\LYA. If $xy=\langle x,y,z\rangle=0$ for all $x,y,z\in A$, then $(A,*,\{\cdot,\cdot,\cdot\})$ becomes a pre-\LYA, which is the underlying algebraic structure of relative Rota-Baxter operators of zero weights. One can see \cite{Sheng Zhao} for more details.
\end{rmk}

\begin{defi}
Let $(A,\cdot_{A},*_{A},\{\cdot,\cdot,\cdot\}_{A},\langle\cdot,\cdot,\cdot\rangle_{A})$ and $(A',\cdot_{A'},*_{A'},\{\cdot,\cdot,\cdot\}_{A'},\langle\cdot,\cdot,\cdot\rangle_{A'})$ be two post-\LYA. A {\bf homomorphism} from $A$ to $A'$ is a linear map $\psi:A\longrightarrow A'$ such that $\psi:(A,\cdot_{A},\langle\cdot,\cdot,\cdot\rangle_{A})\longrightarrow (A,\cdot_{A'},\langle\cdot,\cdot,\cdot\rangle_{A'})$ is a \LYA ~homomorphism and the following equalities hold
\begin{eqnarray*}
\psi(x*_Ay)&=&\psi(x)*_{A'}\psi(y),\\
~\psi(\{x,y,z\}_A)&=&\{\psi(x),\psi(y),\psi(z)\}_{A'},\quad\forall x,y,z\in A.
\end{eqnarray*}
\end{defi}

The following Theorem reveals the fact that a post-\LYA ~gives rise to a \LYA ~structure on its original space.
\begin{thm}\label{thm:sub}
Let $(A,\cdot,*,\{\cdot,\cdot,\cdot\},\langle\cdot,\cdot,\cdot\rangle)$ be a post-\LYA. Then
\begin{itemize}
\item [\rm (i)] the triple $(A,[\cdot,\cdot]_C,\Courant{\cdot,\cdot,\cdot}_C)$ is a \LYA, which is called the {\bf sub-adjacent \LYA} of $(A,\cdot,*,\{\cdot,\cdot,\cdot\},\langle\cdot,\cdot,\cdot\rangle)$ and is denoted by $A^c$, where $[\cdot,\cdot]_C$ and $\Courant{\cdot,\cdot,\cdot}_C$ are given by \eqref{sub1} and \eqref{sub2} respectively;
\item[\rm (ii)] the pair $(L,\huaR)$ is an action of the sub-adjacent \LYA ~$(A,[\cdot,\cdot]_C,\Courant{\cdot,\cdot,\cdot}_C)$ on the \LYA ~$(A,\cdot,\langle\cdot,\cdot,\cdot\rangle)$, where the linear maps $L:A\longrightarrow\gl(A), ~x\mapsto L(x)$ and $\huaR:\otimes^2A\longrightarrow A, ~(x,y)\mapsto \huaR(x,y)$ are defined to be
$$L(x)z=x*z,\quad \huaR(x,y)z=\{z,x,y\},\quad \forall z\in A$$
respectively;
\item[\rm (iii)] the identity map ${\Id}:A\longrightarrow A$ is a relative Rota-Baxter operator of weight $1$ from $(A,\cdot,\langle\cdot,\cdot,\cdot\rangle)$ to $(A,[\cdot,\cdot]_C,\Courant{\cdot,\cdot,\cdot}_C)$.
\end{itemize}
\end{thm}
\begin{proof}
Obviously, the operation $[\cdot,\cdot]_C$ is skew-symmetric, and $\Courant{\cdot,\cdot,\cdot}_C$ is skew-symmetric with respect to the first two variables. For all $x,y,z,w\in A$, by Eq. \eqref{dbra}, we have that
\begin{eqnarray*}
~ &&[[x,y]_C,z]_C+\Courant{x,y,z}_C+c.p.\\
~ &=&[x,y]_C*z-z*(x*y)+z*(y*x)+\{x,y,z\}_D+\{x,y,z\}-\{y,x,z\}+(xy)z+\langle x,y,z\rangle+c.p.\\
~ &=&0,
\end{eqnarray*}
and by \eqref{post1}, we obtain that
\begin{eqnarray*}
~ &&\Courant{[x,y]_C,z,w}_C+c.p.(x,y,z)\\
~ &=&\{[x,y]_C,z,w\}_D+\{[x,y]_C,z,w\}-\{z,[x,y]_C,w\}+\langle xy,z,w\rangle+c.p.(x,y,z)\\
~ &=&0.
\end{eqnarray*}
Thus, Eqs. \eqref{LY1} and \eqref{LY2} with respect to $([\cdot,\cdot]_C,\Courant{\cdot,\cdot,\cdot}_C)$ hold. Similarly, Eqs. \eqref{LY3} and \eqref{fundamental} with respect to $([\cdot,\cdot]_C,\Courant{\cdot,\cdot,\cdot}_C)$ also hold. Thus $(A,[\cdot,\cdot]_C,\Courant{\cdot,\cdot,\cdot}_C)$ is a \LYA. This gives (i).

For all $x,y,z\in A$, we have
\begin{eqnarray*}
~D_{L,\huaR}(x,y)z&=&\huaR(y,x)z-\huaR(x,y)z+[L(x),L(y)]z-L([x,y]_C)z\\
~ &=&\{z,y,x\}-\{z,x,y\}+(y,x,z)-(x,y,z)-(x,y)*z\\
~ &=&\{x,y,z\}_D.
\end{eqnarray*}
In the sequel, we denote by $\huaL:=D_{L,\huaR}$, and thus $\huaL(x,y)z=\{x,y,z\}_D$. Moreover, for all $x,y,z,w\in A$, we also have that
\begin{eqnarray*}
~ &&\huaR([x,y]_C,z)w-\huaR(x,z)L(y)w+\huaR(y,z)L(x)z\\
~ &=&\{w,[x,y]_C,z\}-\{y*w,x,z\}+\{x*w,y,z\}\\
~ &\stackrel{\eqref{post1}}{=}&0,
\end{eqnarray*}
and
\begin{eqnarray*}
~ &&L(\Courant{x,y,z}_C)w-[\huaL(x,y),L(z)]w\\
~ &=&\Courant{x,y,z}_C*w-\{x,y,z*w\}_D+z*\{x,y,w\}_D\\
~ &\stackrel{\eqref{post3}}{=}&0.
\end{eqnarray*}
The left equalities can be obtained similarly. Thus $(L,\huaR)$ is an action of $(A,[\cdot,\cdot]_C,\Courant{\cdot,\cdot,\cdot}_C)$ on $(A,\cdot,\langle\cdot,\cdot,\cdot\rangle)$. This gives (ii).

(iii) follows from the expressions of $[\cdot,\cdot]_C$ and $\Courant{\cdot,\cdot,\cdot}_C$. This finishes the proof.
\end{proof}

\begin{cor}\label{homosub}
Let $\psi:A\longrightarrow A'$ be a homomorphism from a post-\LYA ~$(A,\cdot_{A},*_{A},\\
\{\cdot,\cdot,\cdot\}_{A},\langle\cdot,\cdot,\cdot\rangle_{A})$ to another post-\LYA ~$(A',\cdot_{A'},*_{A'},\{\cdot,\cdot,\cdot\}_{A'},\langle\cdot,\cdot,\cdot\rangle_{A'})$. Then $\psi$ is also a \LYA ~homomorphism between the sub-adjacent \LYA s from $A^c$ to ${A'}^c$.
\end{cor}

The following theorem demonstrates that a post-\LYA ~structure is the underlying algebraic structure of relative Rota-Baxter operators of nonzero weights.
\begin{thm}\label{underlying}
Let $T:\h\longrightarrow\g$ be a relative Rota-Baxter operator of weight $1$ from a \LYA ~$(\h,[\cdot,\cdot]_\h,\Courant{\cdot,\cdot,\cdot}_\h)$ to another \LYA ~$(\g,[\cdot,\cdot]_\g,\Courant{\cdot,\cdot,\cdot}_\g)$ with respect to an action $(\rho,\mu)$. Then  $(\h,\cdot,*,\{\cdot,\cdot,\cdot\},\langle\cdot,\cdot,\cdot\rangle)$ is a post-\LYA, where operations $\cdot,*,\{\cdot,\cdot,\cdot\}$ and $\langle\cdot,\cdot,\cdot\rangle$ are defined to be for all $u,v,w \in \h$,
$$uv=[u,v]_\h,\quad u*v=\rho(Tu)v, \quad \{u,v,w\}=\mu(Tv,Tw)u,\quad \langle u,v,w\rangle=\Courant{u,v,w}_\h.$$
\end{thm}
\begin{proof}
For all $u,v,w \in \h$, we have that
\begin{eqnarray*}
D(Tu,Tv)w&=&\mu(Tv,Tu)w-\mu(Tu,Tv)w+[\rho(Tu),\rho(Tv)]w-\rho([Tu,Tv])w\\
~ &=&\{w,v,u\}-\{w,u,v\}+(v,u,w)-(u,v,w)-(uv)*w\\
~ &=&\{u,v,w\}_D.
\end{eqnarray*}
Moreover, for all $u,v,w,s,t\in \h$, we also have that
\begin{eqnarray*}
~ &&\mu([Tu,Tv],Tt)w-\mu(Tu,Tt)\rho(Tv)w+\mu(Tv,Tt)\rho(Tu)w\\
~ &=&\{w,[u,v]_C,t\}-\{v*w,u,t\}+\{u*w,v,t\},\\
~&&\\
~ &&\mu(Tv,[Tw,Tt])u-\rho(Tw)\mu(Tv,Tt)u+\rho(Tt)\mu(Tv,Tw)u\\
~ &=&\{u,v,[w,t]_C\}-w*\{u,v,t\}+t*\{u,v,w\},\\
~ &&\\
~ &&\rho(\Courant{Tu,Tv,Tw})t-[D(Tu,Tv),\rho(Tw)]t\\
~ &=&\Courant{u,v,w}_C*t-\{u,v,w*t\}+w*\{u,v,t\},\\
\\
~ &&\mu(Tt,Ts)\mu(Tv,Tw)u-\mu(Tw,Ts)\mu(Tv,Tt)u-\mu(Tv,\Courant{Tw,Tt,Ts})u\\
~ &&~~+D(Tw,Tt)\mu(Tv,Ts)u\\
~ &=&\{\{u,v,w\},t,s\}-\{\{u,v,t\},w,s\}+\{w,t\{u,v,s\}\}_D-\{u,v,\Courant{w,t,s}_C\},\\
\\
~ &&[D(Tu,Tv),\mu(Tt,Ts)]w-\mu(\Courant{Tu,Tv,Tt},Ts)w-\mu(Tt,\Courant{Tu,Tv,Ts})w\\
~ &=&\{u,v,\{w,t,s\}\}_D-\{\{u,v,w\}_D,t,s\}-\{w,\Courant{u,v,t}_C,s\}-\{w,t,\Courant{u,v,s}_C\}.
\end{eqnarray*}
Since $(\rho,\mu)$ is a representation of $\g$ on $\h$, Eqs. \eqref{post1}-\eqref{post5} with respect to $(\cdot,*,\{\cdot,\cdot,\cdot\},\langle\cdot,\cdot,\cdot\rangle)$ hold. Moreover, since $(\rho,\mu)$ is also an action, it is easy to see that Eqs. \eqref{post6}-\eqref{post8} with respect to $(\cdot,*,\{\cdot,\cdot,\cdot\},\langle\cdot,\cdot,\cdot\rangle)$ hold. Besides, the triple $(\h,\cdot,\langle\cdot,\cdot,\cdot\rangle)$ is obviously a \LYA. Thus $(\h,\cdot,*,\{\cdot,\cdot,\cdot\},\langle\cdot,\cdot,\cdot\rangle)$ is a post-\LYA. This finishes the proof.
\end{proof}

Combining Theorem \ref{thm:sub} and Theorem \ref{underlying}, we obtain the following corollary directly.
\begin{cor}\label{cor:des}
Let $T:\h\longrightarrow\g$ be a relative Rota-Baxter operator of weight $1$ from a \LYA ~$(\h,[\cdot,\cdot]_\h,\Courant{\cdot,\cdot,\cdot}_\h)$ to another \LYA ~$(\g,[\cdot,\cdot]_\g,\Courant{\cdot,\cdot,\cdot}_\g)$ with respect to an action $(\rho,\mu)$. Then $(\h,[\cdot,\cdot]_T,\Courant{\cdot,\cdot,\cdot}_T)$ is a \LYA, which is called the {\bf descent \LYA}, where $([\cdot,\cdot]_T,\Courant{\cdot,\cdot,\cdot}_T)$ is defined to be
\begin{eqnarray*}
[u,v]_T&=&\rho(Tu)v-\rho(Tv)u+[u,v]_\h,\\
~\Courant{u,v,w}_T&=&D(Tu,Tv)w+\mu(Tv,Tw)u-\mu(Tu,Tw)v+\Courant{u,v,w}_\h,\quad\forall u,v,w\in \h.
\end{eqnarray*}
Thus $T$ is a \LYA ~homomorphism from the descent \LYA ~$(\h,[\cdot,\cdot]_T,\\
\Courant{\cdot,\cdot,\cdot}_T)$ to the \LYA ~$(\g,[\cdot,\cdot]_\g,\Courant{\cdot,\cdot,\cdot}_\g)$.
\end{cor}

\begin{pro}\label{homopost}
Let $T$ and $T'$ be two relative Rota-Bater operators of weight $1$ from a \LYA ~$(\h,[\cdot,\cdot]_\h,\Courant{\cdot,\cdot,\cdot}_\h)$ to another \LYA ~$(\g,[\cdot,\cdot]_\g,\Courant{\cdot,\cdot,\cdot}_\g)$ with respect to an action $(\rho,\mu)$. Let $(\h,\cdot_{T},*_{T},\{\cdot,\cdot,\cdot\}_{T},\langle\cdot,\cdot,\cdot\rangle_{T})$ and $(\h,\cdot_{T'},*_{T'},\{\cdot,\cdot,\cdot\}_{T'},\langle\cdot,\cdot,\cdot\rangle_{T'})$ be the induced post-\LYA s, and $(\psi_\g,\psi_\h)$ a homomorphism from $T'$ to $T$. Then $\psi_\h$ is a homomorphism from the post-\LYA ~$(\h,\cdot_{T'},*_{T'},\{\cdot,\cdot,\cdot\}_{T'},\langle\cdot,\cdot,\cdot\rangle_{T'})$ to the post-\LYA ~$(\h,\cdot_{T},*_{T},\{\cdot,\cdot,\cdot\}_{T},\langle\cdot,\cdot,\cdot\rangle_{T})$.
\end{pro}
\begin{proof}
For all $u,v,w\in \h$, we have
\begin{eqnarray*}
\psi_\h(\{u,v,w\}_{T'})&=&\psi_\h\Big(\mu(T'v,T'w)u\Big)\\
~ &=&\mu\Big(\psi_\g(T'v),\psi_\g(T'w)\Big)\Big(\psi_\h(w)\Big)\\
~ &=&\mu\Big(T\psi_\h(v),T\psi_\h(w)\Big)\Big(\psi_\h(w)\Big)\\
~ &=&\{\psi_\h(u),\psi_\h(v),\psi_\h(w)\}_{T},
\end{eqnarray*}
and
\begin{eqnarray*}
\psi_\h\Big(u*_{T'}v\Big)&=&\psi_\h\Big(\rho(T'u)v\Big)\\
~ &=&\rho(\psi_\g(T'u))\Big(\psi_\h(v)\Big)\\
~ &=&\rho(T\psi_\h(u))\Big(\psi_\h(v)\Big)\\
~ &=&\psi_\h(u)*_T\psi_\h(v)
\end{eqnarray*}
Moreover, since $\psi_\h$ is a homomorphism from $\h$ to $\h$, then we also have that $\psi_\h(u\cdot_{T'}v)=\psi_\h(u)\cdot_{T}\psi_\h(v)$ and $\psi_\h(\langle u,v,w\rangle_{T'})=\langle\psi_\h(u),\psi_\h(v),\psi_\h(w)\rangle_T$ for all $u,v\in \h$. This finishes the  proof.
\end{proof}

\section{Cohomology and deformations of relative Rota-Baxter operators of nonzero weights on Lie-Yamaguti algebras}
\subsection{Cohomology of relative Rota-Baxter operators of nonzero weights on Lie-Yamaguti algebras}
In this subsection, we build the cohomology of relative Rota-Baxter operators of weight $1$ on Lie-Yamaguti algebras and show that the cohomology we build enjoys some funtorial properties.

\subsubsection{Construction of cohomology of relative Rota-Baxter operators of nonzero weights}
Let $(\g,[\cdot,\cdot]_\g,\Courant{\cdot,\cdot}_\g)$ and $(\h,[\cdot,\cdot]_\h,\Courant{\cdot,\cdot}_\h)$ be two Lie-Yamaguti algebras, and $T:\h\longrightarrow\g$ a relative Rota-Baxter operators of weight $1$ with respect to an action $(\rho,\mu)$. Define $\rho_T:\h\longrightarrow\gl(\g)$ and $\mu_T:\otimes^2\h\longrightarrow\gl(\g)$ to be
\begin{eqnarray}
\rho_T(u)(x)&:=&[Tu,x]_\g+T\Big(\rho(x)u\Big),\label{newrep1}\\
~\mu_T(u,v)(x)&:=&\Courant{x,Tu,Tv}_\g-T\Big(D(x,Tu)v-\mu(x,Tv)u\Big), \quad \forall u,v\in \h,~x\in \g.\label{newrep2}
\end{eqnarray}
respectively.

\begin{lem}
With the assumptions above, define $D_T:\wedge^2\h\longrightarrow\gl(\g)$ to be
\begin{eqnarray}
~D_T(u,v)(x)&:=&\Courant{Tu,Tv,x}_\g-T\Big(\mu(Tv,x)u-\mu(Tu,x)v\Big),\quad\forall u,v\in \h,~x\in \g.\label{newrep3}
\end{eqnarray}
Then we have $D_T=D_{\rho_T,\mu_T}$.
\end{lem}
\begin{proof}
For all $u,v\in \h,~x\in \g$, we have
\begin{eqnarray*}
~ D_{\rho_T,\mu_T}(u,v)(x)&\stackrel{\eqref{rep}}{=}&\mu_T(v,u)(x)-\mu_T(u,v)(x)+[\rho_T(u),\rho_T(v)](x)-\rho_T([u,v]_T)(x)\\
~ &\stackrel{\eqref{newrep1},\eqref{newrep2}}{=}&\Courant{x,Tv,Tu}_\g-T\Big(D(x,Tv)u-\mu(x,Tu)v\Big)\\
~ &&-\Courant{x,Tu,Tv}_\g+T\Big(D(x,Tu)v-\mu(x,Tv)u\Big)\\
~ &&+[Tu,[Tv,x]_\g]_\g+T\Big(\rho([Tv,x]_\g)u\Big)\\
~ &&+[Tu,T\Big(\rho(x)v\Big)]_\g+T\Big(\rho(T(\rho(x)v))u\Big)\\
~ &&-[Tv,[Tu,x]_\g]_\g-T\Big(\rho([Tu,x]_\g)v\Big)\\
~ &&-[Tv,T\Big(\rho(x)u\Big)]_\g-T\Big(\rho(T(\rho(x)u))v\Big)\\
~ &&-[T[u,v]_T,x]_\g-T\Big(\rho(x)[u,v]_T\Big)\\
~ &\stackrel{\eqref{chomo1}}=&\Courant{x,Tv,Tu}_\g-T\Big(D(x,Tv)u-\mu(x,Tu)v\Big)\\
~ &&-\Courant{x,Tu,Tv}_\g+T\Big(D(x,Tu)v-\mu(x,Tv)u\Big)\\
~ &&+[Tu,[Tv,x]_\g]_\g+T\Big(\rho([Tv,x]_\g)u\Big)\\
~ &&+T\Big(\rho(Tu)\rho(x)v-\rho\big(T(\rho(x)v)u\big)+[u,\rho(x)v]_\h\Big)+T\Big(\rho(T(\rho(x)v))u\Big)\\
~ &&-[Tv,[Tu,x]_\g]_\g-T\Big(\rho([Tu,x]_\g)v\Big)\\
~ &&-T\Big(\rho(Tv)\rho(x)u-\rho\big(T(\rho(x)u)v\big)+[v,\rho(x)u]_\h\Big)-T\Big(\rho(T(\rho(x)u))v\Big)\\
~ &&-[[Tu,Tv]_\g,x]_\g-T\Big(\rho(x)\big(\rho(Tu)v-\rho(Tv)u+[u,v]_\h\big)\Big)\\
~ &\stackrel{\eqref{LY1},\eqref{*1}}{=}&\Courant{Tu,Tv,x}_\g-T\Big(\mu(Tv,x)u-\mu(Tu,x)v\Big)\\
~ &=&D_T(u,v)(x).
\end{eqnarray*}
This completes the proof.
\end{proof}

\begin{pro}\label{newrepre}
With the assumptions above, then $(\g;\rho_T,\mu_T)$ is a representation of $(\h,[\cdot,\cdot]_T,\Courant{\cdot,\cdot,\cdot}_T)$, where $\rho_T,~\mu_T$, and $D_T$ are given by \eqref{newrep1}-\eqref{newrep3} respectively.
\end{pro}
\begin{proof}
For all $u,v,w\in \h,~x\in \g$, we have
\begin{eqnarray*}
~ &&\mu_T([u,v]_T,w)x-\mu_T(u,w)\rho_T(v)x+\mu_T(v,w)\rho_T(u)x\\
~ &\stackrel{\eqref{newrep1},\eqref{newrep2}}{=}&\Courant{x,T[u,v]_T,Tw}_\g-T\Big(D(x,T[u,v]_T)w-\mu(x,Tw)[u,v]_T\Big)\\
~ &&-\Courant{[Tv,x]_\g,Tu,Tw}_\g+T\Big(D([Tv,x]_\g,Tu)w-\mu([Tv,x]_\g,Tw)u\Big)\\
~ &&-\Courant{T\Big(\rho(x)v\Big),Tu,Tw}_\g+T\Big(D(T(\rho(x)v),Tu)w-\mu(T(\rho(x)v),Tw)u\Big)\\
~ &&+\Courant{[Tu,x]_\g,Tv,Tw}_\g-T\Big(D([Tu,x]_\g,Tv)w-\mu([Tu,x]_\g,Tw)v\Big)\\
~ &&+\Courant{T\Big(\rho(x)u\Big),Tv,Tw}_\g-T\Big(D(T(\rho(x)u),Tv)w-\mu(T(\rho(x)u),Tw)v\Big)\\
~ &\stackrel{\eqref{chomo1}}{=}&\Courant{x,[Tu,Tv]_\g,Tw}_\g-T\Big(D(x,[Tu,Tv]_\g)w-\mu(x,Tw)\big(\rho(Tu)v-\rho(Tv)u+[u,v]_\h\big)\Big)\\
~ &&-\Courant{[Tv,x]_\g,Tu,Tw}_\g+T\Big(D([Tv,x]_\g,Tu)w-\mu([Tv,x]_\g,Tw)u\Big)\\
~ &&-T\Big(D(T(\rho(x)v),Tu)w+\mu(Tu,Tw)\rho(x)v-\mu(T(\rho(x)v),Tw)u+\Courant{\rho(x)v,u,w}_\h\Big)\\
~ &&+T\Big(D(T(\rho(x)v),Tu)w-\mu(T(\rho(x)v),Tw)u\Big)\\
~ &&+\Courant{[Tu,x]_\g,Tv,Tw}_\g-T\Big(D([Tu,x]_\g,Tv)w-\mu([Tu,x]_\g,Tw)v\Big)\\
~ &&+T\Big(D(T(\rho(x)u),Tv)w+\mu(Tv,Tw)\rho(x)u-\mu(T(\rho(x)u),Tw)v+\Courant{\rho(x)u,v,w}_\h\Big)\\
~ &&-T\Big(D(T(\rho(x)u),Tv)w-\mu(T(\rho(x)u),Tw)v\Big)\\
~ &\stackrel{\eqref{RLYb},\eqref{*1},\eqref{*2}}{=}&0,
\end{eqnarray*}
and
\begin{eqnarray*}
~ &&\rho_T(\Courant{u,v,w}_T)(x)-[D_T(u,v),\rho_T(w)](x)\\
~ &\stackrel{\eqref{newrep1},\eqref{newrep3}}{=}&[T\Courant{u,v,w}_T,x]_\g+T\Big(\rho(x)\Courant{u,v,w}_T\Big)\\
~ &&-\Courant{Tu,Tv,[Tw,x]_\g}_\g+T\Big(\mu(Tv,[Tw,x]_\g)u-\mu(Tu,[Tw,x]_\g)v\Big)\\
~ &&-\Courant{Tu,Tv,T\big(\rho(x)w\big)}_\g+T\Big(\mu(Tv,T\big(\rho(x)w\big))u-\mu(Tu,T\big(\rho(x)w\big))v\Big)\\
~ &&+[Tw,\Courant{Tu,Tv,x}_\g]_\g+T\Big(\rho(\Courant{Tu,Tv,x}_\g)w\Big)\\
~ &&-[Tw,T\big(\mu(Tv,x)u\big)]_\g-T\Big(\rho(T\big(\mu(Tv,x)u\big))w\Big)\\
~ &&+[Tw,T\big(\mu(Tu,x)v\big)]_\g+T\Big(\rho(T\big(\mu(Tu,x)v\big))w\Big)\\
~ &\stackrel{\eqref{chomo2}}{=}&[\Courant{Tu,Tv,Tw}_\g,x]_\g+T\Big(\rho(x)\big(D(Tu,Tv)w+\mu(Tv,Tw)u-\mu(Tu,Tw)v+\Courant{u,v,w}_\h\big)\Big)\\
~ &&-\Courant{Tu,Tv,[Tw,x]_\g}_\g+T\Big(\mu(Tv,[Tw,x]_\g)u-\mu(Tu,[Tw,x]_\g)v\Big)\\
~ &&-T\Big(D(Tu,Tv)\rho(x)w+\mu(Tv,T\big(\rho(x)w\big))u-\mu(Tu,T\big(\rho(x)w\big))v+\Courant{u,v,\rho(x)w}_\h\Big)\\
~ &&+T\Big(\mu(Tv,T\big(\rho(x)w\big))u-\mu(Tu,T\big(\rho(x)w\big))v\Big)\\
~ &&-T\Big(\rho(Tw)\big(\mu(Tv,x)u\big)-\rho(T(\mu(Tv,x)u))w+[w,\mu(Tv,x)u]_\h\Big)-T\Big(\rho(T\big(\mu(Tv,x)u\big))w\Big)\\
~ &&+T\Big(\rho(Tw)\big(\mu(Tu,x)v\big)-\rho(T(\mu(Tu,x)v))w+[w,\mu(Tu,x)v]_\h\Big)+T\Big(\rho(T\big(\mu(Tu,x)v\big))w\Big)\\
~ &\stackrel{\eqref{RLYe},\eqref{*1},\eqref{*3}}{=}&0.
\end{eqnarray*}
Other equalities can be obtained similarly. We omit the details.
\end{proof}

\begin{rmk}\label{rmk:coho}
Let $T:\h\longrightarrow\g$ be a relative Rota-Bxter operator of weight $1$ with respect to an action $(\rho,\mu)$, then by Corollary \ref{cor:des}, $T$ induces a sub-adjacent \LYA ~$(\h,[\cdot,\cdot]_T,\Courant{\cdot,\cdot,\cdot}_T)$. Consequently, by Proposition \ref{newrepre}, we obtain a cohomology of this sub-adjacent \LYA ~$(\h,[\cdot,\cdot]_T,\Courant{\cdot,\cdot,\cdot}_T)$ with the coefficients in the representation $(\g;\rho_T,\mu_T)$:
$$\Big(\bigoplus_{p=1}^\infty C_{LieY}^p(\h,\g),\delta^T\Big),$$
where the coboundary map $\delta^T$ is the corresponding coboundary map given by \eqref{cohomology1}-\eqref{cohomology4} with coefficients in the representation $(\g;\rho_T,\mu_T)$.
\end{rmk}

Let $(\g,[\cdot,\cdot]_\g,\Courant{\cdot,\cdot,\cdot}_\g)$ and $(\h,[\cdot,\cdot]_\h,\Courant{\cdot,\cdot,\cdot}_\h)$ be two Lie-Yamaguti algebras. Let $(\rho,\mu)$ be an action of $\g$ on $\h$. Define $\delta:\wedge^2\g\longrightarrow\Hom(\h,\g)$ to be
\begin{eqnarray}
\Big(\delta(x,y)\Big)v:=T\Big(D(x,y)v\Big)-\Courant{x,y,Tv}_\g,\quad\forall x,y\in \g,~v\in \h.\label{0cochain}
\end{eqnarray}

\begin{pro}\label{0co}
With the notations above, $\delta(x,y)$ defined by \eqref{0cochain} is a $1$-cocycle of the sub-adjacent Lie-Yamaguti algebra $(\h,[\cdot,\cdot]_T,\Courant{\cdot,\cdot,\cdot}_T)$ with coefficients in the representation $(\g;\rho_T,\mu_T)$.
\end{pro}
\begin{proof}
We need to show that $\delta_{\rm I}^T\big(\delta(x,y)\big)(u,v)=0$ and that $\delta^T_{\rm II}\big(\delta(x,y)\big)(u,v,w)=0$ for all $u,v,w \in \h$. Indeed, for all $u,v\in \h$, we compute that
\begin{eqnarray*}
~ &&\delta_{\rm I}^T\big(\delta(x,y)\big)(u,v)\\
~ &\stackrel{\eqref{cohomology3}}{=}&\rho_T(u)\delta(x,y)(v)-\rho_T(v)\delta(x,y)(u)-\delta(x,y)[u,v]_T\\
~ &\stackrel{\eqref{newrep1}}{=}&[Tu,T(D(x,y)v)]_\g-[Tu,\Courant{x,y,Tv}_\g]_\g\\
~ &&+T\Big(\rho(T(D(x,y)v))u\Big)-T\Big(\rho\big(\Courant{x,y,Tv}_\g\big)u\Big)\\
~ &&-[Tv,T(D(x,y)u)]_\g+[Tv,\Courant{x,y,Tu}_\g]_\g\\
~ &&-T\Big(\rho(T(D(x,y)u))v\Big)+T\Big(\rho\big(\Courant{x,y,Tu}_\g\big)v\Big)\\
~ &&-T\Big(D(x,y)\big([u,v]_T\big)\Big)+\Courant{x,y,T[u,v]_T}_\g\\
~ &\stackrel{\eqref{chomo1}}{=}&T\Big(\rho(Tu)D(x,y)v-\rho(T(D(x,y)v))u+[u,D(x,y)v]_\h\Big)-[Tu,\Courant{x,y,Tv}_\g]_\g\\
~ &&+T\Big(\rho(T(D(x,y)v))u\Big)-T\Big(\rho\big(\Courant{x,y,Tv}_\g\big)u\Big)\\
~ &&-T\Big(\rho(Tv)D(x,y)u-\rho(T(D(x,y)u))v+[v,D(x,y)u]_\h\Big)+[Tv,\Courant{x,y,Tu}_\g]_\g\\
~ &&-T\Big(\rho(T(D(x,y)u))v\Big)+T\Big(\rho\big(\Courant{x,y,Tu}_\g\big)v\Big)\\
~ &&-T\Big(D(x,y)\big(\rho(Tu)v-\rho(Tv)u+[u,v]_\h\big)\Big)+\Courant{x,y,[Tu,Tv]_\g}_\g\\
~ &\stackrel{\eqref{LY3},\eqref{RLYe},\eqref{*1}}{=}&0.
\end{eqnarray*}
Similarly, we can show that $\delta^T_{\rm II}\big(\delta(x,y)\big)(u,v,w)=0$ for all $u,v,w\in \h$.
\end{proof}

By now, we establish the cohomology of relative Rota-Baxter operator of nonzero weights on Lie-Yamaguti algebra as follows.
Let $T:\h\longrightarrow\g$ be a relative Rota-Baxter operator of weight $1$ from a Lie-Yamaguti algebra $(\h,[\cdot,\cdot]_\h,\Courant{\cdot,\cdot,\cdot}_\h)$ to another \LYA ~$(\g,[\cdot,\cdot]_\g,\Courant{\cdot,\cdot,\cdot}_\g)$ with respect to an action $(\rho,\mu)$. Define the set of $n$-cochains to be
\begin{eqnarray*}
\frkC_T^p(\h,\g)=
\begin{cases}
C_{LieY}^p(\h,\g), & p\geq 1,\\
\wedge^2\g, & p=0.
\end{cases}
\end{eqnarray*}

Define the coboundary map $\partial:\frkC_T^p(\h,\g)\longrightarrow\frkC_T^{p+1}(\h,\g)$ to be
\begin{eqnarray*}
\partial=
\begin{cases}
\delta^T, & p\geq 1,\\
\delta, & p=0,
\end{cases}
\end{eqnarray*}
Then combining Remark \ref{rmk:coho} with Proposition \ref{0co}, we obtain that $\Big(\bigoplus_{n=0}^\infty\frkC_T^n(\h,\g),\partial\Big)$ is a complex. Denote the set of $n$-cochains by $\huaZ_T^n(\h,\g)$, and denote the set of $n$-cobonudaries by $\huaB_T^n(\h,\g)$. The resulting $n$-th cohomology group is given by
$$\huaH_T^n(\h,\g):=\huaZ_T^n(\h,\g)/\huaB_T^n(\h,\g), n\geq 1.$$

\begin{defi}
The cohomology of the cochian complex $\Big(\bigoplus_{n=0}^\infty\frkC_T^n(\h,\g),\partial\Big)$ is called the {\bf cohomology} of the relative Rota-Baxter operator $T$ of weight $1$.
\end{defi}

\subsubsection{Functorial properties of cohomology of relative Rota-Baxter operators of nonzero weight}

In this subsection, we show that the cohomology theory for relative Rota-Baxter operators of nonzero weights has some functorial properties.

\begin{pro}\label{pro1}
Let $T$ and $T'$ be two relative Rota-Baxter operator of weight $1$ from a Lie-Yamaguti algebra $(\h,[\cdot,\cdot]_\h,\Courant{\cdot,\cdot,\cdot}_\h)$ to another Lie-Yamaguti algebra $(\g,[\cdot,\cdot]_\g,\Courant{\cdot,\cdot,\cdot}_\g)$ with respect to an action $(\rho,\mu)$. Let $(\psi_\g,\psi_\h)$ be a homomorphism from $T'$ to $T$. Then
\begin{itemize}
\item [\rm (i)] $\psi_\h$ is a \LYA ~from the descendent \LYA ~$(\h,[\cdot,\cdot]_{T'},\Courant{\cdot,\cdot,\cdot}_{T'})$ of $T'$ to another descent \LYA ~$(\h,[\cdot,\cdot]_{T},\Courant{\cdot,\cdot,\cdot}_{T})$ of $T$;
    \item[\rm (ii)] the induced representation $(\g;\rho_{T'},\mu_{T'})$ of the descendent \LYA ~$(\h,[\cdot,\cdot]_{T'},\\
    \Courant{\cdot,\cdot,\cdot}_{T'})$ and the induced representation $(\g;\rho_{T},\mu_{T})$ of the descendent \LYA ~$(\h,[\cdot,\cdot]_{T},\Courant{\cdot,\cdot,\cdot}_{T})$ satisfy
\begin{eqnarray*}
\psi_\g\circ \rho_{T'}(u)&=&\rho_T(\psi_\h(u))\circ\psi_\g,\\
\psi_\g\circ \mu_{T'}(u,v)&=&\rho_T(\psi_\h(u),\psi_\h(v))\circ\psi_\g,\quad\forall u,v\in \h.
\end{eqnarray*}
That is, the following two diagrams commute:
\[
\xymatrix{
\g \ar[r]^{\psi_\g} \ar[d]_{\rho_{T'}(u)} & \g \ar[d]^{\rho_T(\psi_\h(u))} &\quad\quad\quad\quad &   \g \ar[r]^{\psi_\g} \ar[d]_{\mu_{T'}(u,v)} & \g \ar[d]^{\mu_T(\psi_\h(u),\psi_\h(v))}\\
\g \ar[r]_{\psi_\g}                       & \g,                             &\quad\quad\quad\quad &   \g \ar[r]_{\psi_\g}                        & \g.
}
\]
\end{itemize}
\end{pro}
\begin{proof}
(i) follows from Corollary \ref{homosub} and Proposition \ref{homopost}.

(ii) For all $u,v\in \h$ and $x\in \g$, we compute that
\begin{eqnarray*}
\psi_\g\Big(\rho_{T'}(u)x\Big)&\stackrel{\eqref{newrep1}}{=}&\psi_\g\Big([T'u,x]_\g+T'\big(\rho(x)u\big)\Big)\\
~ &\stackrel{\eqref{v1},\eqref{crho1}}{=}&[\psi_\g(T'u),\psi_\g(x)]_\g+T\Big(\psi_\h\big(\rho(x)u\big)\Big)\\
~ &\stackrel{\eqref{crosshomo2},\eqref{crosshomo3}}{=}&[\psi_\g(T'u),\psi_\g(x)]_\g+T\Big(\rho\big(\psi_\g(x)\big)\big(\psi_\h(u)\big)\Big)\\
~ &\stackrel{\eqref{crho1}}{=}&[T\psi_\h(u),\psi_\g(x)]_\g+T\Big(\rho\big(\psi_\g(x)\big)\big(\psi_\h(u)\big)\Big)\\
~ &\stackrel{\eqref{newrep1}}{=}&\rho_T(\psi_\h(u))\big(\psi_\g(x)\big),
\end{eqnarray*}
and that
\begin{eqnarray*}
~ &&\psi_\g\Big(\mu_{T'}(u,v)x\Big)\\
~ &\stackrel{\eqref{newrep2}}{=}&\psi_\g\Big(\Courant{x,T'u,T'v}_\g-T'\big(D(x,T'u)v-\mu(x,T'v)u\big)\Big)\\
~ &\stackrel{\eqref{v2},\eqref{crho1}}{=}&\Courant{\psi_\g(x),\psi_\g(T'u),\psi_\g(T'v)}_\g-T\psi_\h\Big(D(x,T'u)v-\mu(x,T'v)u\Big)\\
~ &\stackrel{\eqref{crosshomo2},\eqref{crosshomo3}}{=}&\Courant{\psi_\g(x),T\psi_\h(u),T\psi_\h(v)}_\g-T\Big(D(\psi_\g(x),\psi_\g(T'u))(\psi_\h(v))-\mu(\psi_\g(x),\psi_\g(T'v))(\psi_\h(u))\Big)\\
~ &\stackrel{\eqref{crho1}}{=}&\Courant{\psi_\g(x),T\psi_\h(u),T\psi_\h(v)}_\g-T\Big(D(\psi_\g(x),T\psi_\h(u))(\psi_\h(v))-\mu(\psi_\g(x),T\psi_\h(v))(\psi_\h(u))\Big)\\
~ &\stackrel{\eqref{newrep2}}{=}&\mu_T(\psi_\h(u),\psi_\h(v))\big(\psi_\g(x)\big).
\end{eqnarray*}
This finishes the proof.
\end{proof}

Let $T$ and $T'$ be two relative Rota-Baxter operator of weight $1$ from a Lie-Yamaguti algebra $(\h,[\cdot,\cdot]_\h,\Courant{\cdot,\cdot,\cdot}_\h)$ to another Lie-Yamaguti algebra $(\g,[\cdot,\cdot]_\g,\Courant{\cdot,\cdot,\cdot}_\g)$ with respect to an action $(\rho,\mu)$. Let $(\psi_\g,\psi_\h)$ be a homomorphism from $T'$ to $T$ such that $\psi_\h$ is invertible. For $n\geqslant 2$, define a linear map $p:\frkC_{T'}^n(\h,\g)\longrightarrow\frkC_{T}^n(\h,\g)$ to be
$$p(f,g)=\Big(p_{\rm I}(f),p_{\rm II}(g)\Big), \quad \forall (f,g)\in \frkC_{T'}^n(\h,\g)$$
where $p_{\rm I}(f)$ and $p_{\rm II}(g)$ are defined by
\begin{eqnarray}
p_{\rm I}(f)(\frkU_1,\cdots,\frkU_n)&=&\psi_\g\Big(f\big(\psi_\h^{-1}(\frkU_1),\cdots,\psi_\h^{-1}(\frkU_n)\big)\Big),\label{w1}\\
p_{\rm II}(g)(\frkU_1,\cdots,\frkU_n,v)&=&\psi_\g\Big(g\big(\psi_\h^{-1}(\frkU_1),\cdots,\psi_\h^{-1}(\frkU_n),\psi_\h^{-1}(v)\big)\Big),\label{w2}
\end{eqnarray}
respectively.
Here $\frkU_k=u_k\wedge v_k\in \wedge^2\h, ~k=1,2,\cdots,n,~v\in \h$, and we use a notation $\psi_\h^{-1}(\frkU_k)=\psi_\h^{-1}(u_k)\wedge\psi_\h^{-1}(v_k),~k=1,2,\cdots,n.$ In particular, for $n=1$, $p:\frkC_{T'}^1(\h,\g)\longrightarrow\frkC_{T}^1(\h,\g)$ is defined to be
$$p(f)(v)=\psi_\g\Big(f\big(\psi_\h^{-1}(v)\big)\Big), \quad\forall v\in \h, \ f\in \frkC_{T'}^1(\h,\g).$$

\begin{thm}
With the notations above, $p$ is  a cochain map from a cochain $(\oplus_{p=1}^\infty\frkC_{T'}^p(\h,\g),\delta^{T'})$ to $(\oplus_{p=1}^\infty\frkC_{T}^p(\h,\g),\delta^{T})$. Consequently, $p$ induces a homomorphism $p_*:\huaH_{T'}^n(\h,\g)\longrightarrow\huaH_{T}^n(\h,\g)$ between cohomology groups.
\end{thm}
\begin{proof}
For all $(f,g)\in \frkC_{H'}^n(\g,\h)~(n\geq2)$, and for all $\frkU_k=u_k\wedge v_k\in \wedge^2\h,~ k=1,2,\cdots,n,~v\in \h$, on one hand, we have
\begin{eqnarray*}
~ &&\delta_{\rm I}^T\Big(p_{\rm I}(f),p_{\rm II}(g)\Big)(\frkU_1,\cdots,\frkU_n)\\
~ &\stackrel{\eqref{cohomology1}}{=}&(-1)^{n-1}\Big(\rho_T(u_n)p_{\rm II}(g)\big(\frkU_1,\cdots,\frkU_{n-1},v_n\big)\\
~ &&\qquad-\rho_T(v_n)p_{\rm II}(g)\big(\frkU_1,\cdots,\frkU_{n-1},u_n\big)\\
~ &&\qquad-p_{\rm II}\big(\frkU_1,\cdots,\frkU_{n-1},[u_n,v_n]_T\big)\Big)\\
~ &&+\sum_{k=1}^{n-1}(-1)^{k+1}D_T(\frkU_k)p_{\rm I}(f)\big(\frkU_1,\cdots,\widehat{\frkU_k}\cdots,\frkU_n\big)\\
~ &&+\sum_{1\leq k<l\leq n}(-1)^kp_{\rm I}(f)\big(\frkU_1,\cdots,\widehat{\frkU_k}\cdots,\frkU_k\circ \frkU_l,\cdots,\frkU_n\big)\\
~ &\stackrel{\eqref{newrep1}}{=}&(-1)^{n-1}\Big([Tu_n,p_{\rm II}\big(\frkU_1,\cdots,\frkU_{n-1},v_n\big)]_\g+T\big(\rho(p_{\rm II}(g)(\frkU_1,\cdots,\frkU_{n-1},v_n\big))u_n\big)\\
~ &&~~~-[Tv_n,p_{\rm II}\big(\frkU_1,\cdots,\frkU_{n-1},u_n\big)]_\g-T\big(\rho(p_{\rm II}(g)(\frkU_1,\cdots,\frkU_{n-1},u_n\big))v_n\big)\\
~ &&~~~-p_{\rm II}\big(\frkU_1,\cdots,\frkU_{n-1},[u_n,v_n]_T\big)\Big)\\
~ &&+\sum_{k=1}^{n-1}(-1)^{k+1}D_T(\frkU_k)p_{\rm I}(f)\big(\frkU_1,\cdots,\widehat{\frkU_k}\cdots,\frkU_n\big)\\
~ &&+\sum_{1\leq k<l\leq n}(-1)^kp_{\rm I}(f)\big(\frkU_1,\cdots,\widehat{\frkU_k}\cdots,\frkU_k\circ \frkU_l,\cdots,\frkU_n\big)\\
~ &\stackrel{\eqref{w1},\eqref{w2}}{=}&(-1)^{n-1}\Big([Tu_n,\psi_\g\big(g(\psi_\h^{-1}(\frkU_1),\cdots,\psi_\h^{-1}(\frkU_{n-1}),\psi_\h^{-1}(v_n))\big)]_\g\\
~ &&~~~~+T\Big(\rho\big(\psi_\g(g(\psi_\h^{-1}(\frkU_1),\cdots,\psi_\h^{-1}(\frkU_{n-1)},\psi_\h^{-1}(v_n)\big))u_n\Big)\\
~ &&~~~-[Tv_n,\psi_\g\big(g(\psi_\h^{-1}(\frkU_1),\cdots,\psi_\h^{-1}(\frkU_{n-1}),\psi_\h^{-1}(u_n))\big)]_\g\\
~ &&~~~~-T\Big(\rho\big(\psi_\g(g(\psi_\h^{-1}(\frkU_1),\cdots,\psi_\h^{-1}(\frkU_{n-1)},\psi_\h^{-1}(u_n)\big))v_n\Big)\\
~ &&~~~-\psi_\g\big(g(\psi_\h^{-1}(\frkU_1),\cdots,\psi_\h^{-1}(\frkU_{n-1}),\psi_\h^{-1}([u_n,v_n]_T))\big)\Big)\\
~ &&+\sum_{k=1}^{n-1}(-1)^{k+1}D_T(\frkU_k)\psi_\g\Big(f(\big(\psi_\h^{-1}(\frkU_1),\cdots,\widehat{\psi_\h^{-1}(\frkU_k)}\cdots,\psi_\h^{-1}(\frkU_n)\big)\Big)\\
~ &&+\sum_{1\leq k<l\leq n}(-1)^k\psi_\g\Big(f\big(\psi_\h^{-1}(\frkU_1),\cdots,\widehat{\psi_\h^{-1}(\frkU_k)}\cdots,\psi_\h^{-1}(\frkU_k)\circ \psi_\h^{-1}(\frkU_l),\cdots,\psi_\h^{-1}(\frkU_n)\big)\Big),
\end{eqnarray*}
on the other hand, we also have
\begin{eqnarray*}
~ &&p_{\rm I}\Big(\delta_{\rm I}^{T'}(f,g)\Big)(\frkU_1,\cdots,\frkU_n)\\
~ &\stackrel{\eqref{w1}}{=}&\psi_\g\Big(\delta_{\rm I}^{T'}(f,g)(\psi_\h^{-1}(\frkU_1),\cdots,\psi_\h^{-1}(\frkU_n))\Big)\\
~ &\stackrel{\eqref{cohomology1}}{=}&(-1)^{n-1}\Big(\psi_\g\big([T'\psi_\h^{-1}(u_n),g(\psi_\h^{-1}(\frkU_1),\cdots,\psi_\h^{-1}(\frkU_{n-1}),\psi_\h^{-1}(v_n))]_\g\big)\\
~ &&~~~~+\psi_\g\big(T'(\rho(g(\psi_\h^{-1}(\frkU_1),\cdots,\psi_\h^{-1}(\frkU_{n-1}),\psi_\h^{-1}(v_n)))\psi_\h^{-1}(u_n))\big)\\
~ &&~~~~-\psi_\g\big([T'\psi_\h^{-1}(v_n),g(\psi_\h^{-1}(\frkU_1),\cdots,\psi_\h^{-1}(\frkU_{n-1}),\psi_\h^{-1}(u_n))]_\g\big)\\
~ &&~~~~-\psi_\g\big(T'(\rho(g(\psi_\h^{-1}(\frkU_1),\cdots,\psi_\h^{-1}(\frkU_{n-1}),\psi_\h^{-1}(u_n)))\psi_\h^{-1}(v_n))\big)\\
~ &&~~~~-g(\psi_\h^{-1}(\frkU_1),\cdots,\psi_\h^{-1}(\frkU_{n-1}),[\psi_\h^{-1}(u_{n}),\psi_\h^{-1}(v_{n-1})]_{T'})\Big)\\
~ &&+\sum_{k=1}^{n-1}(-1)^{k+1}\psi_\g\Big(D_{T'}(\psi_\h^{-1}(\frkU_k))f\big(\psi_\h^{-1}(\frkU_1),\cdots,\widehat{\psi_\h^{-1}(\frkU_k)}\cdots,\psi_\h^{-1}(\frkU_n)\big)\Big)\\
~ &&+\sum_{1\leq k<l\leq n}(-1)^k\psi_\g\Big(f\big(\psi_\h^{-1}(\frkU_1),\cdots,\widehat{\psi_\h^{-1}(\frkU_k)}\cdots,\psi_\h^{-1}(\frkU_k)\circ \psi_\h^{-1}(\frkU_l),\cdots,\psi_\h^{-1}(\frkU_n)\big)\Big).
\end{eqnarray*}
By the fact that the pair $(\psi_\g,\psi_\h)$ is a homomorphism from $T'$ to $T$ and Proposition \ref{pro1}, we deduce that $\delta_{\rm I}^T\Big(p_{\rm I}(f),p_{\rm II}(g)\Big)=p_{\rm I}\Big(\delta_{\rm I}^{T'}(f,g)\Big)$. Similarly, we also have $\delta_{\rm II}^T\Big(p_{\rm II}(f),p_{\rm II}(g)\Big)=p_{\rm II}\Big(\delta_{\rm II}^{T'}(f,g)\Big)$ for all $(f,g)\in \frkC_{T'}^n(\h,\g)~~(n\geq2)$. Moreover, the case of $n=1$ also holds. We omit the details.
\end{proof}

\subsection{Deformatons of relative Rota-Baxter operators of nonzero weights on Lie-Yamaguti algebras}
In this subsection, we study two kinds of deformations of relative Rota-Baxter operators of weight $1$ on Lie-Yamaguti algebras via the cohomology theory established in the former section.
\subsubsection{Linear deformations of relative Rota-Baxter operators of nonzero weights on Lie-Yamaguti algebras}
In this subsection, we use the cohomology constructed in the former section to characterize the linear deformations of relative Rota-Baxter operators of weight $1$ on Lie-Yamaguti algebras.

\begin{defi}
Let $T:\h\longrightarrow\g$ be a relative Rota-Baxter operator of weight $1$ from a Lie-Yamaguti algebra $(\h,[\cdot,\cdot]_\h,\Courant{\cdot,\cdot,\cdot}_\h)$ to another Lie-Yamaguti algebra $(\g,[\cdot,\cdot]_\g,\Courant{\cdot,\cdot,\cdot}_\g)$  with respect to an action $(\h;\rho,\mu)$. Let $\frkT:\h\longrightarrow\g$ be  a linear map. If $T_t:=T+t\frkT$ is still a relative Rota-Baxter operators of weight $1$ for all $t$, then we say that $\frkT$ generates a {\bf linear deformation} of the relative Rota-Baxter operator $T$ of weight $1$.
\end{defi}

It is direct to deduce that $\frkT$ generates a linear deformation of the relative Rota-Baxter operator $T$ of weight $1$, then $\frkT$ is a $2$-cocycle of $T$.

\begin{defi}
Let $T:\h\longrightarrow\g$ be a relative Rota-Baxter operator of weight $1$ from a Lie-Yamaguti algebra $(\h,[\cdot,\cdot]_\h,\Courant{\cdot,\cdot,\cdot}_\h)$ to another Lie-Yamaguti algebra $(\g,[\cdot,\cdot]_\g,\Courant{\cdot,\cdot,\cdot}_\g)$ with respect to an action $(\rho,\mu)$.
 \begin{itemize}
 \item[(i)] Two linear deformations $T_t^1=T+t\frkT_1$ and $T_t^2=T+t\frkT_2$ are called {\bf equivalent} if there exists an element $\frkX\in \wedge^2\g$ such that $({\Id}_\g+t\frkL(\frkX),{\Id}_\h+tD(\frkX))$ is a homomorphism from $T_t^2$ to $T_t^1$.
 \item[(ii)] A linear deformation $T_t=T+t\frkT$ of the relative Rota-Baxter operator $T$ of weight $1$ is called {\bf trivial} if it is equivalent to $T$.
 \end{itemize}
\end{defi}

If two linear deformations $T_t^2$ and $T_t^1$ are equivalent, then \emptycomment{${\Id}_\g+t\frkL(\frkX)$ is a homomorphism on the Lie-Yamaguti algebra $\g$, which implies that
\begin{eqnarray}
[\Courant{\frkX,y}_\g,\Courant{\frkX,z}_\g]_\g&=&0,\label{nij1}\\
~\Courant{\Courant{\frkX,y}_\g,\Courant{\frkX,z}_\g,t}_\g+\Courant{\Courant{\frkX,y}_\g,z,\Courant{\frkX,t}_\g}_\g+\Courant{y,\Courant{\frkX,z}_\g,\Courant{\frkX,t}_\g}_\g&=&0,\\
~\Courant{\Courant{\frkX,y}_\g,\Courant{\frkX,z}_\g,\Courant{\frkX,t}_\g}_\g&=&0
\end{eqnarray}
for all $y,z,t\in \g$.

Note that by Eqs. \eqref{crosshomo1} and \eqref{crosshomo2}, we obtain that for all $y,z\in \g$,
\begin{eqnarray}
\rho(\Courant{\frkX,y}_\g)D(\frkX)&=&0,\\
~\mu(\Courant{\frkX,y}_\g,z)D(\frkX)+\mu(y,\Courant{\frkX,z}_\g)D(\frkX)+\mu(\Courant{\frkX,y}_\g,\Courant{\frkX,z}_\g)&=&0,\\
~ \mu(\Courant{\frkX,y}_\g,\Courant{\frkX,z}_\g)D(\frkX)&=&0.\label{nij2}
\end{eqnarray}}
 Eq. \eqref{crho1} yields that for all $v\in \h$,
\begin{eqnarray*}
\Big({\Id}_\g+t\frkL(\frkX)\Big)\Big(T+t\frkT_2\Big)v=\Big(T+t\frkT_1\Big)\Big({\Id}_\h+tD(\frkX)\Big)v,
\end{eqnarray*}
which implies that
\begin{eqnarray}
\label{cohomology}\frkT_2v-\frkT_1v=T\Big(D(\frkX)v\Big)-\Courant{\frkX,Tv}_\g=\partial(\frkX)v.
\end{eqnarray}

Thus by \eqref{cohomology}, we have the following theorem.

\begin{thm}
Let $T:\h\longrightarrow\g$ be a relative Rota-Baxter operator of weight $1$ from a Lie-Yamaguti algebra $(\h,[\cdot,\cdot]_\h,\Courant{\cdot,\cdot,\cdot}_\h)$ to another Lie-Yamaguti algebra $(\g,[\cdot,\cdot]_\g,\Courant{\cdot,\cdot,\cdot}_\g)$ with respect to an action $(\rho,\mu)$. If two linear deformations $T_t^1=T+t\frkT_1$ and $T_t^2=T+t\frkT_2$ are equivalent, then $\frkT_1$ and $\frkT_2$ are in the same cohomology class in $\huaH_T^2(\h,\g)$.
\end{thm}

\subsubsection{Higher Order deformations of relative Rota-Baxter operators of nonzero weights on Lie-Yamaguti algebras}
In this subsection, we introduce a special cohomology class associated to an order $n$ deformation of a relative Rota-Baxter operator of weight $1$, and show that a deformation of order $n$ of a a relative Rota-Baxter operator of weight $1$ is extendable if and only if this cohomology class in the second cohomology group is trivial. Thus this kind of special cohomology class is called the obstruction class of an order $n$ deformation being extendable.

 \begin{defi}
 Let $T:\h\longrightarrow\g$ be a relative Rota-Baxter operator of weight $1$ from a Lie-Yamaguti algebra $(\h,[\cdot,\cdot]_\h,\Courant{\cdot,\cdot,\cdot}_\h)$ to another Lie-Yamaguti algebra $(\g,[\cdot,\cdot]_\g,\Courant{\cdot,\cdot,\cdot}_\g)$  with respect to an action $(\rho,\mu)$. If $T_t=\sum_{i=0}^n\frkT_it^i $ with  $\frkT_0=T$, $\frkT_i\in \Hom_{\mathbb K}(\h,\g)$, $i=1,2,\cdots,n$, defines a $\mathbb K[t]/(t^{n+1})$-module from $\h[t]/(t^{n+1})$ to the Lie-Yamaguti algebra $\g[t]/(t^{n+1})$ satisfying
 \begin{eqnarray}
 \label{ordern1}[T_tu,T_tv]_\g&=&T_t\Big(\rho(T_tu)v-\rho(T_tv)u+[u,v]_\h\Big),\\
\label{ordern2}\Courant{T_tu,T_tv,T_tw}_\g&=&T_t\Big(D(T_tu,T_tv)w+\mu(T_tv,T_tw)u-\mu(T_tu,T_tw)v+\Courant{u,v,w}_\h\Big),
 \end{eqnarray}
 for all $u,v,w\in \h$. we say that $T_t$ is an {\bf order $n$ deformation} of $T$.
 \end{defi}

 \begin{rmk}
 The left hand sides of Eqs. \eqref{ordern1} and \eqref{ordern2} hold in the Lie-Yamaguti algebra $\g[t]/(t^{n+1})$ and the right hand sides of  Eqs. \eqref{ordern1} and \eqref{ordern2} make sense since $T_t$ is a $\mathbb K[t]/(t^{n+1})$-module map.
 \end{rmk}

 \begin{defi}
 Let $T:\h\longrightarrow\g$ be a a relative Rota-Baxter operator of weight $1$ from a Lie-Yamaguti algebra $(\h,[\cdot,\cdot]_\h,\Courant{\cdot,\cdot,\cdot}_\h)$ to another Lie-Yamaguti algebra $(\g,[\cdot,\cdot]_\g,\Courant{\cdot,\cdot,\cdot}_\g)$ with respect to an action $(\rho,\mu)$. Let $T_t=\sum_{i=0}^n\frkT_it^i $ be an order $n$ deformation of $T$. If there exists a $1$-cochain $\frkT_{n+1}\in \Hom_{\mathbb K}(\h,\g)$ such that $\widetilde{T_t}=T_t+\frkT_{n+1}t^{n+1}$ is an order $n+1$ deformation of $T$, then we say that $T_t$ is {\bf extendable}.
 \end{defi}

The following theorem is the key conclusion in this section.
 \begin{thm}\label{ob}
Let $T:\h\longrightarrow\g$ be a a relative Rota-Baxter operator of weight $1$ from a Lie-Yamaguti algebra $(\h,[\cdot,\cdot]_\h,\Courant{\cdot,\cdot,\cdot}_\h)$ to another Lie-Yamaguti algebra $(\g,[\cdot,\cdot]_\g,\Courant{\cdot,\cdot,\cdot}_\g)$ with respect to an action $(\rho,\mu)$. Let $T_t=\sum_{i=0}^n\frkT_it^i $ be an order $n$ deformation of $T$. Then $T_t$ is extendable if and only if the cohomology class ~$[\Ob^T]\in \huaH_T^2(\h,\g)$ is trivial, where $\Ob^T=(\Ob_{\rm I}^T,\Ob_{\rm II}^T)\in \frkC_T^2(\h,\g)$ is defined by
\begin{eqnarray*}
\Ob_{\rm I}^T(v_1,v_2)&=&\sum_{i+j=n+1,\atop i,j\geqslant 1}\Big([\frkT_iv_1,\frkT_jv_2]_\g-\frkT_i(\rho(\frkT_jv_1)v_2-\rho(\frkT_jv_2)v_1)\Big),\\
\Ob_{\rm II}^T(v_1,v_2,v_3)&=&\sum_{i+j+k=n+1,\atop i,j,k\geqslant 1}\Big(\Courant{\frkT_iv_1,\frkT_jv_2,\frkT_kv_3}_\g-\frkT_i\big(D(\frkT_jv_1,\frkT_kv_2)v_3\\
~ \nonumber&&+\mu(\frkT_jv_2,\frkT_kv_3)v_1-\mu(\frkT_jv_1,\frkT_kv_3)v_2\big)\Big), \quad \forall v_1,v_2,v_3 \in h.
\end{eqnarray*}
\end{thm}
\begin{proof}
Let $\widetilde{T_t}=\sum_{i=0}^{n+1}\frkT_it^i$ be the extension of $T_t$, then for all $u,v,w \in \h$,
\begin{eqnarray}
\label{n order1}[\widetilde{T_t}u,\widetilde{T_t}v]&=&\widetilde{T_t}\Big(\rho(\widetilde{T_t}u)v-\rho(\widetilde{T_t}v)u+[u,v]_\h\Big),\\
\label{n order2}\Courant{\widetilde{T_t}u,\widetilde{T_t}v,\widetilde{T_t}w}&=&\widetilde{T_t}\Big(D(\widetilde{T_t}u,\widetilde{T_t}v)w+\mu(\widetilde{T_t}v,\widetilde{T_t}w)u
-\mu(\widetilde{T_t}u,\widetilde{T_t}w)v+\Courant{u,v,w}_\h\Big).
\end{eqnarray}
Expanding the Eq.\eqref{n order1} and comparing the coefficients of $t^n$ yields that
\begin{eqnarray*}
\sum_{i+j=n+1,\atop i,j \geqslant 0}\Big([\frkT_iu,\frkT_jv]-\frkT_i\big(\rho(\frkT_ju)v-\rho(\frkT_jv)u\big)-\frkT_{n+1}([u,v]_\h)\Big)=0,
\end{eqnarray*}
which is equivalent to
\begin{eqnarray*}
~ &&\sum_{i+j=n+1,\atop i,j \geqslant 1}\Big([\frkT_iu,\frkT_jv]-\frkT_i\big(\rho(\frkT_ju)v-\rho(\frkT_jv)u\big)\Big)\\
~ &&\qquad+[\frkT_{n+1}u,Tv]+[Tu,\frkT_{n+1}v]
-T\Big(\rho(\frkT_{n+1}u)v-\rho(\frkT_{n+1}v)u\Big)-\frkT_{n+1}([u,v]_T)=0,
\end{eqnarray*}
i.e.,
\begin{eqnarray}
\Ob_{\rm I}^T+\delta_{\rm I}^T(\frkT_{n+1})=0.\label{ob:cocy1}
\end{eqnarray}
Similarly, expanding the Eq.\eqref{n order2} and comparing the coefficients of $t^n$ yields that
\begin{eqnarray*}
~ &&\sum_{i+j+k=n+1,\atop i,j,k\geqslant 1}\Big(\Courant{\frkT_iu,\frkT_jv,\frkT_kw}_\g-\frkT_i\big(D(\frkT_ju,\frkT_kv)w+\mu(\frkT_jv,\frkT_kw)u-\mu(\frkT_ju,\frkT_kw)v\big)\Big)\\
~ &&+\Courant{\frkT_{n+1}u,Tv,Tw}+\Courant{Tu,\frkT_{n+1}v,Tw}+\Courant{Tu,Tv,\frkT_{n+1}w}-T\Big(D(\frkT_{n+1}u,Tv)w+D(Tu,\frkT_{n+1}v)w\\
~ &&+\mu(\frkT_{n+1}v,Tw)u+\mu(Tv,\frkT_{n+1}w)u-\mu(\frkT_{n+1}u,Tw)v-\mu(Tu,\frkT_{n+1}w)v\Big)-\frkT_{n+1}(\Courant{u,v,w}_T)=0,
\end{eqnarray*}
which implies that
\begin{eqnarray}
\Ob_{\rm II}^T+\delta_{\rm II}^T(\frkT_{n+1})=0.\label{ob:cocy2}
\end{eqnarray}
From \eqref{ob:cocy1} and \eqref{ob:cocy2}, we get
$$\Ob^T=-\delta^T(\frkT_{n+1}).$$
Thus, the cohomology class $[\Ob^T]$ is trivial.

Conversely, suppose that the cohomology class $[\Ob^T]$ is trivial, then there exists $\frkT_{n+1}\in \huaC_T^1(V,\g)$, such that ~$\Ob^T=-\delta^T(\frkT_{n+1}).$ Set $\widetilde{T_t}=T_t+\frkT_{n+1}t^{n+1}$. Then for all $0 \leqslant s\leqslant n+1$, ~$\widetilde{T_t}$ satisfies
\begin{eqnarray*}
\sum_{i+j=s}\Big([\frkT_iu,\frkT_jv]-\frkT_i\big(\rho(\frkT_ju)v-\rho(\frkT_jv)u\big)-\frkT_s([u,v]_\h)\Big)=0,\\
\sum_{i+j+k=s}\Big(\Courant{\frkT_iu,\frkT_jv,\frkT_kw}-\frkT_i\big(D(\frkT_ju,\frkT_kv)w+\mu(\frkT_jv,\frkT_kw)u-\mu(\frkT_ju,\frkT_kw)v\big)-\frkT_s(\Courant{u,v,w}_\h)\Big)=0.
\end{eqnarray*}
which implies that $\widetilde{T_t}$ is an order $n+1$ deformation of $T$. Hence it is a extension of $T_t$.
\end{proof}

\begin{defi}
Let $T:\h\longrightarrow\g$ be a a relative Rota-Baxter operator of weight $1$ from a Lie-Yamaguti algebra $(\h,[\cdot,\cdot]_\h,\Courant{\cdot,\cdot,\cdot}_\h)$ to another Lie-Yamaguti algebra $(\g,[\cdot,\cdot]_\g,\Courant{\cdot,\cdot,\cdot}_\g)$ with respect to an action $(\rho,\mu)$, and $T_t=\sum_{i=0}^n\frkT_it^i$ be an order $n$ deformation of $T$. Then the cohomology class $[\Ob^T]\in \huaH_T^2(\h,\g)$ defined in Theorem {\rm \ref{ob}} is called the {\bf obstruction class } of $T_t$ being extendable.
\end{defi}

\begin{cor}
Let $T:\h\longrightarrow\g$ be a a relative Rota-Baxter operator of weight $1$ from a Lie-Yamaguti algebra $(\h,[\cdot,\cdot]_\h,\Courant{\cdot,\cdot,\cdot}_\h)$ to another Lie-Yamaguti algebra $(\g,[\cdot,\cdot]_\g,\Courant{\cdot,\cdot,\cdot}_\g)$ with respect to an action $(\rho,\mu)$. If $\huaH_T^2(\h,\g)=0$, then every $1$-cocycle in $\huaZ_T^1(\h,\g)$ is the infinitesimal of some formal deformation of $T$.
\end{cor}

\vspace{2mm}

 \noindent
 {\bf Acknowledgements.}
Xu was partially supported by NSFC grant 12201253 and Natural Science Foundation of Jiangsu Province BK20220510.
Qiao was partially supported by NSFC grant 11971282.
\vspace{2mm}
\emptycomment{
\noindent\textbf{Data Availability Statement.} Data sharing is not applicable to this article as no new data were created
or analyzed in this study.
\vspace{2mm}

\noindent\textbf{Conflict of interest.} The authors declared that they have no conflicts of interest
to this work.}

\end{document}